\documentclass[a4paper,11pt, reqno]{amsart}
\usepackage{amssymb,amsthm,amsmath}

\usepackage{enumerate}
\usepackage{ifthen}
\usepackage{graphicx} 
\usepackage[colorlinks,citecolor=blue,hypertexnames=false]{hyperref}

 \numberwithin{equation}{section}
\newtheorem{thm}{Theorem}[section]
\newtheorem{cor}[thm]{Corollary}
\newtheorem{lem}[thm]{Lemma}
\newtheorem{prop}[thm]{Proposition}
\newtheorem{defn}[thm]{Definition}

\allowdisplaybreaks[2]

\theoremstyle{definition}
\newtheorem{rmk}[thm]{Remark}

{\qed\bigskip}

\newcounter{alphabet}

\ifx\undefined\bysame
\newcommand{\bysame}{\leavevmode\hbox to3em{\hrulefill}\,}
\fi

\title[Fractional damped wave equation on compact Lie groups]
{Nonlinear fractional damped wave equation on compact Lie groups}
\author{Aparajita Dasgupta} 
\address{Aparajita Dasgupta, Assistant Professor \endgraf Department of Mathematics
	\endgraf Indian Institute of Technology  Delhi
	\endgraf Delhi, 110016  India.} 
\email{adasgupta@maths.iitd.ac.in}
\author{Vishvesh Kumar} 

\address{Vishvesh Kumar, Ph. D.  \endgraf Department of Mathematics: Analysis, Logic and Discrete Mathematics
	\endgraf Ghent University
	\endgraf Krijgslaan 281, Building S8,	B 9000 Ghent,
	Belgium.} 
\email{vishveshmishra@gmail.com}

\author{Shyam Swarup Mondal} 

\address{Shyam Swarup Mondal  \endgraf Department of Mathematics
	\endgraf Indian Institute of Technology Delhi
	\endgraf Delhi, 110 016, India.} 
\email{mondalshyam055@gmail.com}

\thanks{ The first and third authors were supported by Core Research Grant(RP03890G), Science and Engineering Research Board (SERB), DST, India.  The second author was supported  by the FWO Odysseus 1 grant G.0H94.18N: Analysis and Partial Differential Equations, the Methusalem programme of the Ghent University Special Research Fund (BOF) (Grant number 01M01021), and by FWO Senior Research Grant G011522N.}
\date{\today}
\begin{document}
	\allowdisplaybreaks

	\begin{abstract} 
		
		 In this paper, we deal with the initial value fractional damped wave equation on $G$, a compact Lie group,  with power-type nonlinearity. The aim of this manuscript is twofold. First,  using the Fourier analysis on compact Lie groups, we prove a local in-time existence result in the energy space for the fractional damped wave equation on $G$. Moreover, a finite time blow-up result is established under certain conditions on the initial data. In the next part of the paper,  we consider fractional wave equation with lower order terms, that is, damping and mass with the same power type nonlinearity on compact Lie groups, and prove the global in-time existence of small data solutions in the energy evolution space.

	\end{abstract}

	\maketitle
	\tableofcontents 
	\section{Introduction}
	The study of partial differential equations is indeed one of the fundamental tools for understanding and modeling natural and real-world phenomena. Fractional differential operators are nonlocal operators that are considered as a generalization of classical differential operators of arbitrary non-integer orders.
	For the last few decades, the study of  partial differential equations involving nonlocal operators have gained a considerable amount of interest and have become one of the essential topics in mathematics and its applications. Many physical phenomena in engineering,    quantum field theory, astrophysics,  biology, materials, control theory, and other sciences can be successfully described by models utilizing mathematical tools from fractional calculus \cite{Neww1, NN1,NN3,NN4, f4}.    In particular, the fractional Laplacian is represented as the infinitesimal generator of stable radially symmetric Lévy processes  \cite{EE1}. 
	For other exciting models related to fractional differential equations, we refer to the reader  \cite{f3, f10,f11,EE24} to mention only a few of many recent publications.

	In recent years, due to the nonlocal nature of the fractional derivatives, considerable attention has been devoted to various models involving fractional Laplacian and nonlocal operators by several researchers.  	There is a vast literature available involving the fractional Laplacian on the Euclidean framework, which is difficult to mention; we refer to important papers     \cite{EE3, EE4, EE5, EE6, EE10, EE24, EE27, EE31} and the references therein.  Here we would like to point out that the fractional Laplacian operator $(-\Delta)^{\alpha}$ can be reduced to the classical Laplace operator $-\Delta$ as $\alpha \rightarrow 1$. We refer to \cite{EE24} for more details. In particular,  many interesting results in some classical elliptic problems have been extended in the fractional Laplacian setting, see 	\cite{EE9}.
	
	For the classical semilinear damped wave equation in $\mathbb{R}^n$, the global existence or a  blow-up result depending on the critical exponent has been studied in \cite{IKeta and Tanizawa, Matsumura, Zhang, Todorova}. We refer to the excellent book    \cite{Rei}  for global in-time small data solutions for the semilinear damped wave equation on the Euclidean framework.  
	
	The study of the semilinear damped wave equation has also been extended in the non-Euclidean framework. Several papers have studied linear PDE in non-Euclidean structures in the last decades.   For example, the semilinear wave equation with or without damping has been investigated for the   Heisenberg group \cite{24,30}.    In the case of graded groups, we refer to the recent works \cite{gra1, gra2, gra3}.  	Concerning the damped wave equation on compact Lie groups, we refer to \cite{27, 28,31,garetto, BKM22}.   Particularly,  the author in \cite{27} studied semilinear damped wave equation with power type nonlinearity $|u|^p$ on compact Lie groups and proved a local in-time existence result in the energy space via Fourier analysis on compact Lie groups. He also derived a blow-up result for the semilinear Cauchy problem for any $p > 1$.   Also,     considering the semilinear wave equation with damping and mass with power nonlinearity $|u|^p$ on compact Lie groups and without any lower bounds for $p > 1$,  the author proved the global in time existence of small data solutions in the evolution energy space in   \cite{28}.
	For the study of semilinear wave equation of general compact manifolds, we refer to the seminal works \cite{BGT04, KAP95} where the global in-time solution were investigated by establishing famous Strichartz type estimates. Recently, the wave equation were also explored in the noncompact manifolds setup, see \cite{AZ22, Zhang20, Zhang21, SSW19} and reference therein. 
	
	Then, an interesting and viable problem is to study  the fractional wave equation  (\ref{eq0010}) and (\ref{2number311})  of order $\alpha$ with $ 0 < \alpha <  1$, with power-type nonlinearity. In \cite{Ahmad15}, the authors have investigated the nonexistence of global weak solutions to the nonlinear fractional wave equation with power type nonlinearlity on the Heisenberg group. In the setting for compact Lie groups, we have recently started a systematic study of the nonlinear fractional wave equation on compact Lie groups. This work is a continuation of our previous work \cite{shyamm}. To state our problem, let  $G$ be a compact Lie group with normalized Haar measure $dx$ and let  $\mathcal{L}$ be the Laplace-Beltrami operator on $G$ (which also coincides with the Casimir element of the universal enveloping algebra of Lie algebra of $G$). For $0<\alpha<1$,  we consider the following two Cauchy problems for the fractional wave equation  with power type nonlinearity, namely, with damping term,
	\begin{align} \label{eq0010}
		\begin{cases}
			\partial^2_tu+(-\mathcal{L})^\alpha u+	\partial_tu =|u|^p, & x\in G,t>0,\\
			u(0,x)=\varepsilon u_0(x),  & x\in G,\\ \partial_tu(x,0)=\varepsilon u_1(x), & x\in G,
		\end{cases}
	\end{align}
	and with damping and positive mass,
	\begin{align} \label{2number311}
		\begin{cases}
			\partial^2_tu+\left( -\mathcal{L}\right)^\alpha u+b\partial_{t} u+m^2u=|u|^p, & x\in G, t>0,\\
			u(0,x)=u_0(x),  & x\in G,\\ \partial_tu(x,0)=u_1(x), & x\in G,
		\end{cases}
	\end{align}
	where  $ p > 1,b,m^2$ are positive constants and $\varepsilon$ is a positive constant describing the smallness of the Cauchy data. Here  for the moment, we assume that  $u_{0}$ and $ u_{1}$ are  taken from the energy space $ H_{\mathcal{L}}^\alpha(G)$ (see \eqref{sob} for the definition) and $ L^2(G)$,  respectively.

	This paper investigates a  finite time blow-up result for solutions to the fractional damped wave equation involving the Laplace-Beltrami operator on compact Lie groups under a suitable sign assumption for the initial data. Moreover,   we show that the presence of a positive damping term and a positive mass term in the   Cauchy problem completely reverses the scenario, i.e.,  we prove the global existence of small data solutions for the fractional wave equation with damping and mass. More preciously, using the Gagliardo-Nirenberg type inequality (in order to handle power nonlinearity in $L^2(G))$ and Fourier analysis on compact Lie groups, we prove the local well-posedness of the   Cauchy problem (\ref{eq0010}) in the energy evolution space  $\mathcal{C}\left([0, T], H_{\mathcal{L}}^{\alpha}( {G})\right) \cap \mathcal{C}^{1}\left([0, T], L^{2}( {G})\right)$ and  the  global in time existence of small data solutions for the        Cauchy problem (\ref{2number311}).

	\subsection{Main results}
	We denote $L^{q}(G), 1 \leq  q<\infty$, the space of $q$-integrable functions on the compact Lie group $G$ concerning the normalized Haar measure $dx$ on $G$   and essentially bounded for $q=\infty$ throughout the paper.   For $s>0$ and $q \in(1, \infty)$, the fractional Sobolev space $H_{\mathcal{L} }^{ s, q}(G)$ of order $\alpha$ is defined as  
	\begin{align}\label{sob}
		H_{\mathcal{L}}^{s, q}(G) \doteq\left\{f \in L^{q}(G):(-\mathcal{L})^{s / 2} f \in L^{q}(G)\right\},
	\end{align}
	endowed with the norm $\|f\|_{H_{\mathcal{L}}^{s, q}(G)} =: \|f\|_{L^{q}(G)}+\left\|(-\mathcal{L})^{s / 2} f\right\|_{L^{q}(G)}$.  We simply denote   the Hilbert space $H_{\mathcal{L}}^{s, 2}(G)$ by $H_{\mathcal{L}}^{s}(G)$.   
	
	By employing noncommutative   Fourier analysis for compact Lie groups, our first result concerning   $L^2$-decay estimates for the solution of the linear version of the Cauchy problem (\ref{eq0010}) (when $f=0$)   is   stated in the following proposition.  
	\begin{prop}\label{thm11}
		Let $0<\alpha <1$.  Suppose that $(u_0, u_1)\in H_{\mathcal{L}}^\alpha(G) \times  L^2(G)$ and $u\in\mathcal{C}([0,\infty),H_{\mathcal{L}}^\alpha(G))\cap \mathcal{C}^1([0,\infty),L^2(G))$ be the solution to the homogeneous Cauchy problem
		\begin{align}\label{eq1}
			\begin{cases}
				\partial^2_tu+(-\mathcal{L})^\alpha u+\partial_tu =0, & x\in G,~t>0,\\
				u(0,x)=u_0(x),  & x\in G,\\ \partial_tu(x,0)=u_1(x), & x\in G.
			\end{cases}
		\end{align}
		Then, $u$ satisfies the following $L^2( G)-L^2( G)$ estimates
		\begin{align}\label{111111}
			\| u(t,\cdot)\|_{L^2( G)}  &\lesssim( \left\|u_{0}\right\|_{L^{2}(G)}+t\left\|u_{1}\right\|_{L^{2}(G)}),\\\nonumber
			\left\|(-\mathcal{L})^{\alpha / 2} u(t, \cdot)\right\|_{L^{2}(G)}&\lesssim (1+t)^{-\frac{1}{2}} (\left\|u_{0}\right\|_{H_{\mathcal{L}}^{{\alpha }}(G)}^{2}+\left\|u_{1}\right\|_{L^{2}(G)}^{2}),\\\nonumber
			\|\partial_tu(t,\cdot)\|_{L^2( G)}&\lesssim (1+t)^{-1} (\left\|u_{0}\right\|_{H_{\mathcal{L}}^{{\alpha }}(G)}^{2}+\left\|u_{1}\right\|_{L^{2}(G)}^{2}).
		\end{align}
		for any $t\geq 0$.
	\end{prop}
	%
	%
	

	Next we prove the local well-posedness of the   Cauchy problem   (\ref{eq0010})  in the energy evolution space  $\mathcal C\left([0,T],  H^\alpha_{\mathcal{L}}(G)\right)\cap\mathcal C^1\left([0,T],L^2(G)\right)$.  In this case, a Gagliardo-Nirenberg type inequality (proved in \cite{Gall}) will be used to estimate the power nonlinearity in $L^2(G)$.	 Indeed, we have the following local existence result.
	\begin{thm}\label{thm22}
		Let $0<\alpha <1$ and let  $G$ be a compact connected Lie group with the topological dimension $n.$ Assume that $n\geq 2[\alpha]+2$. Suppose that $(u_0, u_1)\in H^\alpha_{\mathcal L}(G) \times  L^2(G)$      and $p>1$ such that $p\leq\frac{n}{n-2\alpha}.$ Then there exists $T=T(\varepsilon)>0$ such that the Cauchy problem (\ref{eq0010})   admits a uniquely determined mild solution $$u\in \mathcal{C}([0,T],H^\alpha_{\mathcal L}(G))\cap \mathcal{C}^1([0,T],L^2(G)).$$
	\end{thm}

	\begin{rmk}
		Note that  the  restriction $p\leq\frac{n}{n-2\alpha }$   and $n \geq 2[\alpha]+2$  in the above theorem is necessary in order to apply Gagliardo-Nirenberg type inequality.
	\end{rmk}		
	
	Our next result is about the non-existence of global in-time solutions to (\ref{eq0010}) for any $p > 1$  regardless of the size of initial data. Before stating the blow-up result, we first introduce a suitable notion of energy solutions for the   Cauchy problem (\ref{eq0010}).
	\begin{defn}\label{eq332}
		Let $0<\alpha <1$ and $\left(u_{0}, u_{1}\right) \in H_{\mathcal{L}}^{\alpha}(G) \times L^{2}(G)$. For  any $T>0,$ we say that
		$$
		u \in \mathcal{C}\left([0, T), H_{\mathcal{L}}^{\alpha}(G)\right) \cap \mathcal{C}^{1}\left([0, T), L^{2}(G)\right) \cap L_{\text {loc }}^{p}([0, T) \times G)
		$$
		is an energy solution on $[0, T)$ to (\ref{eq0010}) if $u$ satisfies  the following  integral relation:
		\begin{align}\label{eq011}\nonumber
			&\int_{ {G}} \partial_{t} u(t, x) \phi(t, x)  {d} x-\int_{ {G}} u(t, x) \partial_{s} \phi(t, x) {d} x  +\int_{ {G}} u(t, x)   \phi (t, x) {d} x \\\nonumber
			&  +\varepsilon \int_{G} u_{0}(x) \partial_{s}\phi(0, x) \;d x -\varepsilon \int_{G} u_{1}(x) \phi(0, x) \;d x -\varepsilon \int_{G} u_{0}(x) \phi(0, x) \;d x \\&+\int_{0}^{t} \int_{G} u(s, x)\left(\partial^2_s\phi(s, x) +(-\mathcal{L})^\alpha \phi(s, x)+\partial_s\phi(s, x) \right) d x  ds
			=\int_{0}^{t} \int_{G}|u(s, x)|^{p} \phi(s, x) dxds
		\end{align}
		for any $\phi \in \mathcal{C}_{0}^{\infty}([0, T) \times G)$ and   any $t \in(0, T)$.
	\end{defn}  
	\begin{thm}\label{f6}
		Let  $0<\alpha <1$, $p>1$,  and let $\left(u_{0}, u_{1}\right) \in H_{\mathcal{L}}^{\alpha}(G) \times L^{2}(G)$ be nonnegative and nontrivial functions. Suppose 
		$$u \in \mathcal{C}\left([0, T), H_{\mathcal{L}}^{\alpha}(G)\right) \cap \mathcal{C}^{1}\left([0, T), L^{2}(G)\right) \cap L_{ {loc}}^{p}([0, T) \times G)$$ be an energy solution to the Cauchy problem   (\ref{eq0010})    with lifespan $T=T(\varepsilon)$. Then there exists a   constant $\varepsilon_{0}=\varepsilon_{0}\left(u_{0}, u_{1}, p\right)>0$ such that for any $\varepsilon \in\left(0, \varepsilon_{0}\right],$ the energy solution $u$ blows up in finite time.
		Furthermore, the lifespan $T$ satisfies the following   estimates
		\begin{align}\label{eq1112}
			T(\varepsilon) \leq 
			C \varepsilon^{1-p}.
		\end{align}
		
	\end{thm}
	%
	\begin{rmk}
		\begin{itemize}
			\item[(i)] Here we note that the fractional Laplace-Beltrami operator  $(-\mathcal{L})^{\alpha}$ gives the classical Laplace-Beltrami operator    $-\mathcal{L}$ as $\alpha \rightarrow 1$ and all our results coincides with the results proved for the   Cauchy problem for the semilinear damped wave equation on compact Lie groups  in \cite{27}.
			
			\item[(ii)] 	From  Theorem \ref{f6}  one can see that  the sharp lifespan estimates  	for local in-time solutions to (\ref{eq0010}) is    independent of $\alpha, 0<\alpha<1$.  Thus,  for any $0<\alpha<1$, the lifespan estimates  for solutions  to the Cauchy problem for the fractional wave equation  (\ref{eq0010})  will be the same as the  sharp lifespan estimates for 
			the semilinear wave equation on compact Lie group $G$ proved in  \cite{27}.
		\end{itemize}

	\end{rmk}

	In the next part of the paper, we study the global existence of small data solutions for the nonlinear fractional wave equation with damping and mass and involving power type nonlinearity. More preciously,   we consider the   Cauchy problem  (\ref{2number311}), i.e.,
	\begin{align*} 
		\begin{cases}
			\partial^2_tu+\left( -\mathcal{L}\right)^\alpha u+b\partial_{t} u+m^2u=|u|^p, & x\in G, t>0,\\
			u(0,x)=u_0(x),  & x\in G,\\ \partial_tu(x,0)=u_1(x), & x\in G,
		\end{cases}
	\end{align*}
	where  $ p > 1$, $b,m^2$ are positive constants, $u_{0}(x)$ and $u_{1}(x)$ are two given functions on $G$.

	First, we prove the following $L^2$-decay estimates with exponential decay rates related to the time variable for the solution of the homogeneous Cauchy problem (\ref{2number311}) (when $f=0$).
	\begin{prop}\label{2thm11}
		Let $0<\alpha <1$.  Suppose that $(u_0, u_1)\in H_{\mathcal{L}}^\alpha(G) \times  L^2(G)$ and $u\in\mathcal{C}([0,\infty),H_{\mathcal{L}}^\alpha(G))\cap \mathcal{C}^1([0,\infty),L^2(G))$ be the solution to the homogeneous Cauchy problem
		\begin{align}\label{2eq1}
			\begin{cases}
				\partial^2_tu+(-\mathcal{L})^\alpha u+b\partial_{t} u+m^2u=0, & x\in G,~t>0,\\
				u(0,x)=u_0(x),  & x\in G,\\ \partial_tu(x,0)=u_1(x), & x\in G.
			\end{cases}
		\end{align}
		Then, $u$ satisfies the following $L^2( G)-L^2( G)$ estimates
		\begin{align}\label{2111111}
			\| u(t,\cdot)\|_{L^2( G)}  &\lesssim C A_{b, m^2}(t)( \left\|u_{0}\right\|_{L^{2}(G)}+t\left\|u_{1}\right\|_{L^{2}(G)}),\\\nonumber
			\left\|(-\mathcal{L})^{\alpha / 2} u(t, \cdot)\right\|_{L^{2}(G)}&\lesssim C A_{b, m^2}(t)   (\left\|u_{0}\right\|_{H_{\mathcal{L}}^{{\alpha }}(G)}^{2}+\left\|u_{1}\right\|_{L^{2}(G)}^{2}),\\\nonumber
			\|\partial_tu(t,\cdot)\|_{L^2( G)}&\lesssim  C A_{b, m^2}(t)   (\left\|u_{0}\right\|_{H_{\mathcal{L}}^{{\alpha }}(G)}^{2}+\left\|u_{1}\right\|_{L^{2}(G)}^{2}).
		\end{align}
		for any $t\geq 0$, where  $C$ is a positive multiplicative constant  and the decay function $A_{b, m^2}(t)$ is given by
		$$
		A_{b, m^2}(t) \doteq\left\{\begin{array}{ll}
			{e}^{-\frac{b}{2} t} & \text { if } b^2<4 m^2, \\
			(t+1)  {e}^{-\frac{b}{2} t} & \text { if } b^2=4 m^2, \\
			{e}^{\left(-\frac{b}{2}+\sqrt{\frac{b^2}{4}-m^2}\right) t} & \text { if } b^2>4 m^2.
		\end{array}\right.
		$$
	\end{prop}
	Using these  above $L^2$-decay estimates, we will prove the global existence of small data solutions  to the  nonlinear fractional   Cauchy problem   (\ref{2number311})  in the energy evolution space  $\mathcal C\left([0,\infty),  H^\alpha_{\mathcal{L}}(G)\right)\cap\mathcal C^1\left([0,\infty),L^2(G)\right)$.  In this case, a Gagliardo-Nirenberg type inequality (proved in \cite{Gall}) will be used to estimate the power nonlinearity in $L^2(G)$.	The following result is about the global existence of the mild solution of the Cauchy problem (\ref{2number311}). For the definition of the mild solution, see subsection \ref{sec4}.

	\begin{thm}\label{2thm22}
		Let $0<\alpha <1$ and let  $G$ be a compact connected Lie group with the topological dimension $n.$ Assume that $n\geq 2[\alpha]+2$. Suppose that $(u_0, u_1)\in H^\alpha_{\mathcal L}(G) \times  L^2(G)$      and $p>1$ such that $p\leq\frac{n}{n-2\alpha}.$ Then there exists $\varepsilon_0>0$ such that for any  $\|(u_0, u_1)\|_{ H^\alpha_{\mathcal L}(G) \times  L^2(G)}\leq \varepsilon_{0}$,  	the Cauchy problem (\ref{2number311})   admits a uniquely determined mild solution $$u\in \mathcal{C}([0,\infty),H^\alpha_{\mathcal L}(G))\cap \mathcal{C}^1([0,\infty),L^2(G)).$$
		
	\end{thm}
	
	\begin{rmk}
		Here we note that the fractional Laplace-Beltrami operator  $(-\mathcal{L})^{\alpha}$ can be reduced to the classical Laplace-Beltrami operator    $-\mathcal{L}$ as $\alpha \rightarrow 1$ and  Proposition \ref{2thm11} and  Theorem \ref{2thm22} coincides with the results proved for the   Cauchy problem for the fractional wave equation with damping and mass on compact Lie groups in \cite{28}.
	\end{rmk}
	\begin{rmk}
		We note that in the statement of Theorem \ref{2thm22}, the  restriction on the upper bound for the exponent $p$ which is $p\leq\frac{n}{n-2\alpha }$   is necessary in order to apply Gagliardo-Nirenberg type inequality (\ref{eq34}) in  (\ref{f}). Also, the other restriction $n \geq 2[\alpha]+2$ is made to fulfill the assumptions for the employment of such inequality.
	\end{rmk}	
	
	Before studying the nonhomogeneous Cauchy problem (\ref{eq0010}) and \eqref{2number311} we first deal with the corresponding homogeneous problem, i.e., when $f=0$. Particularly,  using the group Fourier transform with respect to the spatial variable,  we determine  $L^{2}-L^{2} $ estimates for the solution of the homogeneous fractional damped wave equation on the compact Lie group $G$.  
	Once we have these estimates,  applying a Gagliardo-Nirenberg type inequality on compact Lie groups \cite{27, 28, 31} (see also \cite{Gall}  for  Gagliardo-Nirenberg type inequality on a more general frame of connected Lie groups),   we prove the local well-posedness result for  (\ref{eq0010}) and the global in time solution for  (\ref{2number311}).


	Apart from the introduction, this paper is organized as follows. In Section \ref{sec2},  we recall the Fourier analysis on compact Lie groups which will be used frequently throughout the paper for our approach. In Section \ref{sec3},   first,
	we show an appropriate decomposition of the propagators for the nonlinear equation in the Fourier space. Further,  by recalling the notion of mild solutions in our framework,  we prove Theorem \ref{thm22}, the local existence result,  by deriving some  $L^{2}-L^{2}$ estimates for the solution of the homogeneous fractional wave equation on the compact Lie group $G$.   
	Moreover,  under certain conditions on the initial data, a finite time blow-up result is established.    In Section \ref{sec6}, we prove Theorem \ref{2thm22}, the global existence for the mild solution,   by deriving some  $L^{2}-L^{2}$ estimates for the solution of the homogeneous fractional wave equation with damping and mass (\ref{2number311}) on the compact Lie group $G$.

	\section{Preliminaries: Analysis on compact Lie groups} \label{sec2}
	In this section, we recall some basics of Fourier analysis on compact Lie groups to make the manuscript self-contained. A complete account of the representation theory of the compact Lie groups can be found in \cite{garetto, RT13, RuzT}. However, we mainly adopt the notation and terminology given in \cite{RuzT}.
	
	\subsection{Notations} 
	Throughout the article,  we use the following notations:  
	
	\begin{itemize}
		\item $f \lesssim g:$\,\,There exists a positive constant $C$ (whose value may change from line to line in this manuscript) such that $f \leq C g.$
		\item $G:$ Compact Lie group.
		\item $dx:$ The normalized Haar measure on the compact group $G.$
		\item $\mathcal{L}:$ The Laplace-Beltrami operator on $G.$
		\item $\mathbb{C}^{d \times d}:$ The set of matrices with complex entries of order $d.$
		\item $ \operatorname{Tr}(A)=\sum_{j=1}^{d} a_{j j}:$ The trace  of the matrix $A=\left(a_{i j}\right)_{1 \leq i, j \leq  d} \in \mathbb{C}^{d \times d}.$
		\item $I_{d} \in \mathbb{C}^{d \times d}:$  The identity matrix of order $d.$
	\end{itemize}
	
	\subsection{Representation theory on compact Lie groups}
	Let us first recall the definition of a representation of a compact group $G.$ A unitary representation  of $G$ is a pair $(\xi, \mathcal{H})$ such that the map $\xi:G \rightarrow U(\mathcal{H}),$ where $U(\mathcal{H})$ denotes the set of  unitary operators on complex Hilbert space $\mathcal{H},$  such that it satisfies following properties:
	\begin{itemize}
		\item The map $\xi$ is a group homomorphism, that is, $\xi(x y)=\xi(x)\xi(y).$
		\item The mapping $\xi:G \rightarrow U(\mathcal{H})$ is continuous with respect to strong operator topology (SOT) on $U(\mathcal{H}),$  that is, the map $g \mapsto \xi(g)v$ is continuous for every $v \in \mathcal{H}.$ 
	\end{itemize}
	The Hilbert space $\mathcal{H}$ is called the representation space. To avoid any confusion, we  represent    a representation $(\xi, \mathcal{H})$ of $G$ by $\xi.$ 
	Two unitary representations $\xi, \eta$ of ${G}$ are called  equivalent if there exists an unitary operator, namely intertwiner,  $T$ such that $T \xi(x)=\eta(x) T$ for any $x \in {G}$. 
	An intertwiner is an irreplaceable tool in the theory of representation of compact groups and is helpful in the classification of representation.
	A (linear) subspace $V \subset \mathcal{H}$ is said to be invariant under the unitary representation $\xi$ of $G$ if $\xi(x) V \subset V$, for any $x \in {G}$. An irreducible unitary representation $\xi$ of $G$ is a representation such that the only closed and $\xi$-invariant subspaces of $\mathcal{H}$  are trivial once, that is, $\{0\}$ and the full space $ \mathcal{H}$. 
	
	The set of all equivalence classes  $[\xi]$  of continuous irreducible unitary representations of $G$ is denoted by $\widehat{G}$ and called the unitary dual of $G.$ Since $G$ is compact, $\widehat{G}$ is a discrete set. It is known that an irreducible unitary representation $\xi$ of $G$ is finite-dimensional, i.e., the Hilbert space $\mathcal{H}$ is finite-dimensional, say, $d_\xi$. Therefore, if we  choose a basis $\mathfrak{B}:=\{e_1,e_2,\ldots, e_{d_\xi}\}$ for the representation space $\mathcal{H}$ of $\xi$, we can identify $\mathcal{H}$ as $\mathbb{C}^{d_\xi}$ and consequently, we can view $\xi$ as a matrix-valued function $\xi: G \rightarrow U(\mathbb{C}^{d_{\xi} \times d_{\xi}})$, where $U(\mathbb{C}^{d_{\xi} \times d_{\xi}})$ denotes the space of all unitary matrices. The matrix coefficients $\xi_{ij}$ of the representation $\xi$ with respect to $\mathfrak{B}$ are given by $\xi_{ij}(x):=\langle \xi(x) e_j, e_i \rangle$, for all $i, j \in \{1,2, \ldots, d_\xi\}.$  It follows from the Peter-Weyl theorem that the set
	$$
	\left\{\sqrt{d_{\xi}} \xi_{i j}: 1 \leq i, j \leq d_{\xi},[\xi] \in \widehat{G}\right\}
	$$
	forms an orthonormal basis of $L^{2}(G)$.

	\subsection{Fourier analysis on compact Lie groups } Let $G$ be a compact Lie group.
	The group Fourier transform of $f \in L^1(G)$ at $\xi\in \widehat{G},$ denoted by $\widehat{f}(\xi),$ is defined by
	$$
	\widehat{f}(\xi):=\int_{G} f(x) \xi(x)^{*} d x,
	$$
	where $dx$ is the normalized Haar measure on $G$. It is apparent from the definition that  $\widehat{f}(\xi)$ is matrix-valued and therefore, this definition can be interpreted in weak sense, i.e.,  for $u,v \in \mathcal{H},$ 
	$$ \langle	\widehat{f}(\xi)u, v \rangle:=\int_{G} f(x) \langle \xi(x)^{*}u, v \rangle d x.$$
	
	It follows from  the Peter-Weyl theorem that, for every $f \in L^2(G),$ we  have the following Fourier series representation:
	$$
	f(x)=\sum_{[\xi] \in \widehat{G}} d_{\xi} \operatorname{Tr}(\xi(x) \widehat{f}(\xi)).
	$$
	
	The Plancherel identity for the group Fourier transform on  $G$ takes the following  form
	\begin{align}\label{eq002}
		\|f\|_{L^{2}(G)}=\left(\sum_{[\xi] \in \widehat{G}} d_{\xi}\|\widehat{f}(\xi)\|_{\mathrm{HS}}^{2}\right)^{1 / 2}:=\|\widehat{f}\|_{\ell^2(\widehat{G})},
	\end{align}
	where $\|\cdot\|_{\mathrm{HS}}$ denotes the Hilbert-Schmidt norm of a matrix $A:=(a_{ij}) \in \mathbb{C}^{d_
		\xi \times d_\xi}$ defined as
	$$	\|A\|_{\mathrm{HS}}^{2}=\operatorname{Tr}\left( A A^{*}\right)=\sum_{i, j=1}^{d_{\xi}}|a_{ij}|^2.$$
	We would like to emphasize here that the  Plancherel identity is one of the crucial tools to establish $L^2$-estimates of the solution to PDEs.

	Let $\mathcal{L}$ be the  Laplace-Beltrami operator on $G$. 	It is important to understand the action of the group Fourier transform on the Laplace–Beltrami operator $\mathcal{L}$ for developing the machinery of the proofs. For  $[\xi] \in \widehat{{G}}$,   the matrix elements   $\xi_{i j}$,  are  the eigenfunctions of $\mathcal{L}$ with the same   eigenvalue $-\lambda_{\xi}^{2}$. In other words, we have,    for any $ x \in {G},$
	$$
	-\mathcal{L} \xi_{i j}(x)=\lambda_{\xi}^{2} \xi_{i j}(x),   \qquad \text{for all } i, j \in\left\{1, \ldots, d_{\xi}\right\}.
	$$
	The symbol $\sigma_{\mathcal{L}}$ of  the  Laplace-Beltrami operator $\mathcal{L}$ on $G$ is given by 
	\begin{align}\label{symbol}
		\sigma_{\mathcal{L}}(\xi)=-\lambda_{\xi}^{2} I_{d_{\xi}},
	\end{align}
	for any $[\xi] \in \widehat{{G}}$ and   therefore, the following holds:  $$\widehat{\mathcal{L} f}(\xi)=\sigma_{\mathcal{L}}(\xi) \widehat{f}(\xi)=-\lambda_{\xi}^{2} \widehat{f}(\xi)$$ for any $[\xi] \in \widehat{ G}$. 
	
	For $s>0,$ the   Sobolev space $H_{\mathcal{L}}^s\left(G\right)$ of order $s$ is defined as follows: 
	$$H_{\mathcal{L}}^s(G):=\left\{u \in L^{2}(G):\|u\|_{H_{\mathcal{L}}^s(G)}<+\infty\right\},$$ where $\|u\|_{H_{\mathcal{L}}^s(G)}=\|u\|_{L^{2}(G)}+\left\|(-\mathcal{L})^{s / 2} u\right\|_{L^{2}({G})}$ and    $(-\mathcal{L})^{s / 2} $  is    defined in terms of the group Fourier transform by the follwoing formula
	$$(-\mathcal{L})^{s / 2} f :=\mathcal{F}^{-1}\left(\lambda_{\xi}^{s }(\mathcal{F} f)\right),  	\quad  \text{for all $[\xi] \in \widehat{{G}}$}.$$

	

	Further,  using Plancherel identity,  for any $s>0$, we have that
	$$
	\left\|(-\mathcal{L})^{s / 2} f\right\|_{L^{2}({G})}^{2}=\sum_{[\xi] \in \widehat{{G}}} d_{\xi} \lambda_{\xi}^{2 s}\|\widehat{f}(\xi)\|_{\mathrm{HS}}^{2}
	.	 	$$

	
	\section{A local existance result}\label{sec3}
	In this section,     we study the local well-posedness of the   Cauchy problem   (\ref{2number31}), i.e., 
	\begin{align*}  
		\begin{cases}
			\partial^2_tu+(-\mathcal{L})^\alpha u+	\partial_tu =|u|^p, & x\in G,t>0,\\
			u(0,x)=\varepsilon u_0(x),  & x\in G,\\ \partial_tu(x,0)=\varepsilon u_1(x), & x\in G,
		\end{cases}
	\end{align*}
	where   $u_{0}(x)$ and $u_{1}(x)$ are two given functions on $G$  and   $\varepsilon$ is a positive constant describing the smallness of the Cauchy data. 
	
	\subsection{Fourier multiplier expressions and $L^2(G)-L^2(G)$ estimates 
	} \label{sec3.1}
	In this subsection, we derive $L^2(G)– L^2(G)$ estimates for the solutions to the homogeneous problem     (\ref{eq1}).
	We  employ the group Fourier transform on the compact group $G$ with respect to  the space variable $x$  together with the Plancherel identity in order to      estimate     $L^2$-norms of  $u(t, ·), (-\mathcal{L})^{\frac{\alpha}{2}}u(t, \cdot)$, and $\partial_{t}u(t, ·)$. 
	
	Let $u$ be a solution to (\ref{eq1}). Let $\widehat{u}(t, \xi)=(\widehat{u}(t, \xi)_{kl})_{1\leq k, l\leq d_\xi}\in \mathbb{C}^{d_\xi\times d_\xi}, [\xi]\in\widehat{ G}$ denote the Fourier transform of $u$  with respect to the $x $ variable. Invoking the group Fourier transform with respect to $x$ on   (\ref{eq1}), we deduce that $\widehat{u}(t, \xi)$ is  a solution to the following  Cauchy problem for the system of ODE's (with the size of the system that depends on the representation $\xi$)
	\begin{align}\label{eq6661}
		\begin{cases}
			\partial^2_t\widehat{u}(t,\xi)+(-	\sigma_{\mathcal{L}}(\xi))^\alpha \widehat{u}(t,\xi) +\partial_t \widehat{u}(t,\xi)=0,& [\xi]\in\widehat{ G},~t>0,\\ \widehat{u}(0,\xi)=\widehat{u}_0(\xi), &[\xi]\in\widehat{ G},\\ \partial_t\widehat{u}(0,\xi)=\widehat{u}_1(\xi), &[\xi]\in\widehat{ G},
		\end{cases} 
	\end{align}
	where  $\sigma_{\mathcal{L}}$	is the symbol of  the  operator operator $\mathcal{L}$.  Using the identity (\ref{symbol}),  the   system  (\ref{eq6661}) can be written in the form of $d_\xi^2$ independent   ODE's, namely,
	\begin{align}\label{eqq7}
		\begin{cases}
			\partial^2_t\widehat{u}(t,\xi)_{kl}+	\partial_t\widehat{u}(t,\xi)_{kl}+ \lambda_\xi^{2\alpha }  \widehat{u}(t,\xi)_{kl}= 0,& [\xi]\in\widehat{ G},~t>0,\\ \widehat{u}(0,\xi)_{kl}=\widehat{u}_0(\xi)_{kl}, &[\xi]\in\widehat{ G},\\ \partial_t\widehat{u}(0,\xi)_{kl}=\widehat{u}_1(\xi)_{kl}, &[\xi]\in\widehat{ G},
		\end{cases}
	\end{align}
	for all $k,l\in\{1,2,\ldots,d_\xi\}.$ 	Then, the characteristic equation of (\ref{eqq7}) is given by
	\[\lambda^2+	\lambda+\lambda_\xi^{2\alpha } =0,\]
	and consequently the characteristic roots are   $\lambda=-\frac{1}{2}\pm \frac{\sqrt{1-4\lambda_\xi^{2\alpha}}}{2}  $. 
	%
	Thus the solution to the   homogeneous  problem   (\ref{eqq7}) is given by 
	\begin{align}\label{number2}\nonumber
		\widehat{u}(t,\xi)_{kl}&=e^{-\frac{t}{2}}A_0(t, \xi) \widehat{u}_0(\xi)_{kl}+e^{-\frac{t}{2}}A_1(t, \xi) \left(\widehat{u}_1(\xi)_{kl}+\frac{1}{2}\widehat{u}_0(\xi)_{kl}\right)\\
		&=e^{-\frac{t}{2}}\left[A_0(t, \xi)+\frac{A_1(t, \xi)}{2}\right] \widehat{u}_0(\xi)_{kl}+e^{-\frac{t}{2}}A_1(t, \xi)  \widehat{u}_1(\xi)_{kl},
	\end{align}
	where
	\begin{align}\label{number1}
		A_0(t, \xi)=\begin{cases}
			\cosh \left(\frac{1}{2}\sqrt{ 1-4\lambda_\xi^{2\alpha}}~~t \right)& \text{if } 4\lambda_\xi^{2\alpha}< 1,\\
			1& \text{if } 4\lambda_\xi^{2\alpha}=1,\\
			\cos\left(\frac{1}{2}\sqrt{ 4\lambda_\xi^{2\alpha}-1}~~t \right)& \text{if } 4\lambda_\xi^{2\alpha}> 1,\\
	\end{cases} 	\end{align}
	and
	\begin{align}\label{number3}
		A_1(t, \xi)=\begin{cases}
			\frac{	2\sinh \left(\frac{1}{2}\sqrt{ 1-4\lambda_\xi^{2\alpha}}~~t \right)}{\sqrt{ 1-4\lambda_\xi^{2\alpha}}}& \text{if } 4\lambda_\xi^{2\alpha}< 1,\\
			t& \text{if } 4\lambda_\xi^{2\alpha}=1,\\
			\frac{	\sin\left(\frac{1}{2}\sqrt{ 4\lambda_\xi^{2\alpha}-1}~~t \right)}{\sqrt{ 4 \lambda_\xi^{2\alpha}-1}}& \text{if } 4\lambda_\xi^{2\alpha}> 1.
		\end{cases} 
	\end{align}
	We 	notice that $A_0(t, \xi) = \partial_t A_1(t, \xi)$ for any $[\xi] \in \widehat{ G}$ and
	\begin{align}\label{number22}
		\partial_{t}	\widehat{u}(t,\xi)_{kl}=-e^{-\frac{t}{2}}A_1(t, \xi)\lambda_\xi^{2\alpha} \widehat{u}_0(\xi)_{kl}+e^{-\frac{t}{2}} \left[ A_0(t, \xi)-\frac{1}{2}A_1(t, \xi)\right] \widehat{u}_1(\xi)_{kl}.
	\end{align}

	To simplify the presentation, we introduce the following partition of the unitary dual $\widehat{ G}$ as:
	\begin{align*}
		\mathcal{R}_1&=\{[\xi]\in\widehat{ G}:0\leq  \lambda_\xi^{2\alpha }<\frac{1}{16}\}, \\
		\mathcal{R}_2&=\{[\xi]\in\widehat{ G}:\lambda_\xi^{2\alpha }\geq \frac{1}{16}\}. 
	\end{align*}

	Note that the choice of $ \frac{1}{16}$ as a threshold in the previous definitions is irrelevant since our goal is to separate 	$0$ (which is an eigenvalue for the continuous irreducible unitary representation $1 : x \in G \to  1 \in  \mathbb{C}$)  from the other eigenvalues. Now we        estimate     $L^2$-norms of  $u(t, ·), (-\mathcal{L})^{\frac{\alpha}{2}}u(t, \cdot)$, and $\partial_{t}u(t, ·)$.

	\noindent\textbf{Estimate on $\mathcal{R}_1:$} 
	In this case, 
	$ |A_0(t, \xi)|\leq \cosh  \frac{t}{2}$ and $ |A_1(t, \xi)|\leq \sin  \frac{t}{2}$. Therefore from \eqref{number2}, we have
	\begin{align}\label{eqd1}
		| \widehat{u}(t,\xi)_{kl}| \lesssim   | \widehat{u}_0(\xi)_{kl}| +| \widehat{u}_1(\xi)_{kl}| .
	\end{align}
	Again for $[\xi] \in \mathcal{R}_1$, we have
	\begin{align*}
		A_0(t, \xi)+\frac{A_1(t, \xi)}{2}& = \frac{e^{\frac{1}{2}\sqrt{ 1-4\lambda_\xi^{2\alpha}}~~t }+e^{-\frac{1}{2}\sqrt{ 1-4\lambda_\xi^{2\alpha}}~~t }}{2}+ \frac{e^{\frac{1}{2}\sqrt{ 1-4\lambda_\xi^{2\alpha}}~~t }-e^{-\frac{1}{2}\sqrt{ 1-4\lambda_\xi^{2\alpha}}~~t }}{2\sqrt{ 1-4\lambda_\xi^{2\alpha}}}\\&=\left(
		\frac{1}{2}+\frac{1}{4\sqrt{ 1-4\lambda_\xi^{2\alpha}}} \right)e^{\frac{1}{2}\sqrt{ 1-4\lambda_\xi^{2\alpha}}~~t }+\left(
		\frac{1}{2}-\frac{1}{2\sqrt{ 1-4\lambda_\xi^{2\alpha}}} \right)e^{-\frac{1}{2}\sqrt{ 1-4\lambda_\xi^{2\alpha}}~~t }\\
		&\approx  
		\left(
		\frac{1}{2}+\frac{1}{4\sqrt{ 1-4\lambda_\xi^{2\alpha}}} \right)e^{\frac{1}{2}\sqrt{ 1-4\lambda_\xi^{2\alpha}}~~t }-\frac{\lambda_\xi^{2\alpha}}{\sqrt{ 1-4\lambda_\xi^{2\alpha}}}  e^{-\frac{1}{2}\sqrt{ 1-4\lambda_\xi^{2\alpha}}~~t }.
	\end{align*}
	Thus, from \eqref{number1} we deduce that
	\begin{align*} 
		\widehat{u}(t,\xi)_{kl} 
		&\approx e^{-\frac{t}{2}}\left[
		\left(
		\frac{1}{2}+\frac{1}{4\sqrt{ 1-4\lambda_\xi^{2\alpha}}} \right)e^{\frac{1}{2}\sqrt{ 1-4\lambda_\xi^{2\alpha}}~~t }-\frac{\lambda_\xi^{2\alpha}}{\sqrt{ 1-4\lambda_\xi^{2\alpha}}}  e^{-\frac{1}{2}\sqrt{ 1-4\lambda_\xi^{2\alpha}}~~t }\right] \widehat{u}_0(\xi)_{kl}\\&\quad+e^{-\frac{t}{2}} \left[\frac{e^{\frac{1}{2}\sqrt{ 1-4\lambda_\xi^{2\alpha}}~~t }-e^{-\frac{1}{2}\sqrt{ 1-4\lambda_\xi^{2\alpha}}~~t }}{2\sqrt{ 1-4\lambda_\xi^{2\alpha}}}\right] \widehat{u}_1(\xi)_{kl},
	\end{align*}
	and therefore, 
	\begin{align*} 
		|	\widehat{u}(t,\xi)_{kl} |
		&\lesssim e^{-\frac{t}{2} } \Bigg[  e^{\frac{1}{2}\sqrt{ 1-4\lambda_\xi^{2\alpha}}~~t } \left(\left|\widehat{u}_0(\xi)_{k \ell}\right|+\left|\widehat{u}_1(\xi)_{k \ell}\right|\right) \\&\quad\qquad\qquad\qquad
		+\frac{e^{-\frac{1}{2}\sqrt{ 1-4\lambda_\xi^{2\alpha}}~~t }}{\sqrt{ 1-4\lambda_\xi^{2\alpha}} } \left(\lambda_\xi^{2\alpha} \left|\widehat{u}_0(\xi)_{k \ell}\right|+\frac{1}{2}\left|\widehat{u}_1(\xi)_{k \ell}\right|\right)\Bigg]\\
		&\lesssim e^{-\frac{t}{2} }  \left(\left|\widehat{u}_0(\xi)_{k \ell}\right|+\left|\widehat{u}_1(\xi)_{k \ell}\right|\right) \left[  e^{\frac{1}{2}\sqrt{ 1-4\lambda_\xi^{2\alpha}}~~t }  
		+\frac{e^{-\frac{1}{2}\sqrt{ 1-4\lambda_\xi^{2\alpha}}~~t }}{\sqrt{ 1-4\lambda_\xi^{2\alpha}} }  \right] \\
		&\lesssim e^{-\frac{t}{2}+\frac{1}{2}\sqrt{ 1-4\lambda_\xi^{2\alpha}}~~t }  \left(\left|\widehat{u}_0(\xi)_{k \ell}\right|+\left|\widehat{u}_1(\xi)_{k \ell}\right|\right) \left[ 1
		+\frac{e^{-\sqrt{ 1-4\lambda_\xi^{2\alpha}}~~t }}{\sqrt{ 1-4\lambda_\xi^{2\alpha}} }  \right] \\
		&\approx  e^{-\frac{t}{2}+\frac{1}{2}( 1-2\lambda_\xi^{2\alpha})t}   \left(\left|\widehat{u}_0(\xi)_{k \ell}\right|+\left|\widehat{u}_1(\xi)_{k \ell}\right|\right) \left[ 1
		+\frac{e^{-\sqrt{ 1-4\lambda_\xi^{2\alpha}}~~t }}{\sqrt{ 1-4\lambda_\xi^{2\alpha}} }  \right] \\
		&\lesssim  e^{ -\lambda_\xi^{2\alpha}t}   \left(\left|\widehat{u}_0(\xi)_{k \ell}\right|+\left|\widehat{u}_1(\xi)_{k \ell}\right|\right)  .
	\end{align*}
	This implies using AM-GM inequality  that
	\begin{align}\label{number23}
		\lambda_{\xi}^{2\alpha} \left|\widehat{u}(t, \xi)_{k \ell}\right|^2  \lesssim \lambda_{\xi}^{2\alpha}  {e}^{- 2\lambda_{\xi}^{2\alpha} t}\left(\left|\widehat{u}_0(\xi)_{k \ell}\right|^2 +\left|\widehat{u}_1(\xi)_{k \ell}\right|^2\right) \lesssim(1+t)^{-1}\left(\left|\widehat{u}_0(\xi)_{k \ell}\right|^2+\left|\widehat{u}_1(\xi)_{k \ell}\right|^2\right).
	\end{align}
	We note  that, for $[\xi] \in \mathcal{R}_1$,  we have
	\begin{align*}
		A_0(t, \xi)-\frac{1}{2}A_1(t, \xi)&=\frac{e^{\frac{1}{2}\sqrt{ 1-4\lambda_\xi^{2\alpha}}~~t }+e^{-\frac{1}{2}\sqrt{ 1-4\lambda_\xi^{2\alpha}}~~t }}{2}- \frac{e^{\frac{1}{2}\sqrt{ 1-4\lambda_\xi^{2\alpha}}~~t }-e^{-\frac{1}{2}\sqrt{ 1-4\lambda_\xi^{2\alpha}}~~t }}{2\sqrt{ 1-4\lambda_\xi^{2\alpha}}}\\&=\left(
		\frac{1}{2}-\frac{1}{2\sqrt{ 1-4\lambda_\xi^{2\alpha}}} \right)e^{\frac{1}{2}\sqrt{ 1-4\lambda_\xi^{2\alpha}}~~t }+\left(
		\frac{1}{2}+\frac{1}{2\sqrt{ 1-4\lambda_\xi^{2\alpha}}} \right)e^{-\frac{1}{2}\sqrt{ 1-4\lambda_\xi^{2\alpha}}~~t }\\
		&\approx  
		-\frac{\lambda_\xi^{2\alpha}}{\sqrt{ 1-4\lambda_\xi^{2\alpha}}}  e^{\frac{1}{2}\sqrt{ 1-4\lambda_\xi^{2\alpha}}~~t }+\left(
		\frac{1}{2}+\frac{1}{2\sqrt{ 1-4\lambda_\xi^{2\alpha}}} \right)e^{-\frac{1}{2}\sqrt{ 1-4\lambda_\xi^{2\alpha}}~~t }.
	\end{align*}
	Therefore, using it in \eqref{number22} for $[\xi] \in \mathcal{R}_1$, we get
	\begin{align}\label{number27}\nonumber
		\left|\partial_t \widehat{u}(t, \xi)_{k \ell}\right| & \lesssim \lambda_{\xi}^{2\alpha} {e}^{-\lambda_{\xi}^{2\alpha} t}\left(\left|\widehat{u}_0(\xi)_{k \ell}\right|+\left|\widehat{u}_1(\xi)_{k \ell}\right|\right)+ {e}^{-t}\left(\left|\widehat{u}_0(\xi)_{k \ell}\right|+\left|\widehat{u}_1(\xi)_{k \ell}\right|\right) \\
		& \lesssim(1+t)^{-1}\left(\left|\widehat{u}_0(\xi)_{k \ell}\right|+\left|\widehat{u}_1(\xi)_{k \ell}\right|\right).
	\end{align}
	
	\noindent\textbf{Estimate on $\mathcal{R}_2:$} When $ \frac{1}{16}\leq \lambda_\xi^{2\alpha }<\frac{1}{4}$,  by following the similar calculation, there exists a suitable  positive constant $c_1$     independent of $[\xi]$ such that 
	\begin{align}\label{eqc1}
		| \widehat{u}(t,\xi)_{kl}| \lesssim e^{-c_1t}\left[| \widehat{u}_0(\xi)_{kl}| +| \widehat{u}_1(\xi)_{kl}| \right].
	\end{align}
	When $   \lambda_\xi^{2\alpha }\geq \frac{1}{4}$,  it is easy to note that
	$
	| A_0(t,\xi)| \leq1 $ and $| {A}_1(t,\xi)| \leq t.
	$
	Therefore from \eqref{number2}, there exists a suitable  $c_2>0$     independent of $[\xi]$ such that 
	\begin{align}\label{eqc2}\nonumber
		| \widehat{u}(t,\xi)_{kl}| &\leq  e^{-\frac{t}{2}}  \widehat{u}_0(\xi)_{kl}+t e^{-\frac{t}{2}}  \left(\widehat{u}_1(\xi)_{kl}+\frac{1}{2}\widehat{u}_0(\xi)_{kl}\right)\\\nonumber
		&		\lesssim (1+t) e^{-\frac{t}{2}}   \left[| \widehat{u}_0(\xi)_{kl}| +| \widehat{u}_1(\xi)_{kl}| \right]\\
		&	\lesssim  e^{-c_2 t}   \left[| \widehat{u}_0(\xi)_{kl}| +| \widehat{u}_1(\xi)_{kl}| \right].
	\end{align}
	Thus from  (\ref{eqc1})  and \eqref{eqc2},  we have
	\begin{align}\label{eqd2}
		| \widehat{u}(t,\xi)_{kl}|  	\lesssim  e^{-c t}   \left[| \widehat{u}_0(\xi)_{kl}| +| \widehat{u}_1(\xi)_{kl}| \right].
	\end{align}		
	where $c$ is a suitable  positive constant    independent of $[\xi]$.
	
	Moreover,  for $[\xi] \in \mathcal{R}_2$, it follows that 
	\begin{align}\label{number24}
		\lambda_{\xi}^\alpha\left|\widehat{u}(t, \xi)_{k \ell}\right| \lesssim  {e}^{-c t}\left(\lambda_{\xi}^\alpha\left|\widehat{u}_0(\xi)_{k \ell}\right|+\left|\widehat{u}_1(\xi)_{k \ell}\right|\right),
	\end{align}
	for a suitable positive constant $c$.  
	
	On the other hand, for $[\xi] \in \mathcal{R}_2$,  we get the estimate
	\begin{align}\label{number26}
		\left|\partial_t \widehat{u}(t, \xi)_{k \ell}\right| \lesssim  {e}^{-c t}\left(\lambda_{\xi}^\alpha \left|\widehat{u}_0(\xi)_{k \ell}\right|+\left|\widehat{u}_1(\xi)_{k \ell}\right|\right),
	\end{align}
	where $c>0$ is a suitable constant.

	\textbf{Estimate for $\|u(t, \cdot )\|_{L^{2}(G)}$:}
	Using the Plancherel formula along with  the equations (\ref{eqd1}) and (\ref{eqd2}), it follows that 
	\begin{align}\label{L2}
		\|u(t, \cdot)\|_{L^{2}(G)}^{2}&=\sum_{[\xi] \in \widehat{G}} d_{\xi} \sum_{k, \ell=1}^{d_{\xi}}\left|\widehat{u}(t, \xi)_{k \ell}\right|^{2} \nonumber \\
		&=\sum_{[\xi] \in \mathcal{R}_1} d_{\xi} \sum_{k, \ell=1}^{d_{\xi}}\left|\widehat{u}(t, \xi)_{k \ell}\right|^{2}  +\sum_{[\xi] \in \mathcal{R}_2} d_{\xi} \sum_{k, \ell=1}^{d_{\xi}}\left|\widehat{u}(t, \xi)_{k \ell}\right|^{2} \nonumber\\
		&	\lesssim \sum_{[\xi] \in \mathcal{R}_1} d_{\xi} \sum_{k, \ell=1}^{d_{\xi}} \left(\left|\widehat{u}_{0}(\xi)_{k \ell}\right|^{2}+\left|\widehat{u}_{1}(\xi)_{k \ell}\right|^{2}\right) \nonumber \\&\qquad +\sum_{[\xi] \in \mathcal{R}_2} d_{\xi} \sum_{k, \ell=1}^{d_{\xi}} e^{-2ct}\left(\left|\widehat{u}_{0}(\xi)_{k \ell}\right|^{2}+\left|\widehat{u}_{1}(\xi)_{k \ell}\right|^{2}\right)\nonumber\\
		&	\lesssim \sum_{[\xi] \in \widehat{G}} d_{\xi} \sum_{k, \ell=1}^{d_{\xi}}\left(\left|\widehat{u}_{0}(\xi)_{k \ell}\right|^{2}+\left|\widehat{u}_{1}(\xi)_{k \ell}\right|^{2}\right) \nonumber \\
		&=\left\|u_{0}\right\|_{L^{2}(G)}^{2}+\left\|u_{1}\right\|_{L^{2}(G)}^{2} .
	\end{align}
	
	\textbf{Estimate for $\left\|(-\mathcal{L})^{\alpha / 2} u(t, \cdot )\right\|_{L^{2}(G)}$:}
	Using the Plancherel formula, we  get
	\begin{align}\label{f1}
		\left\|(-\mathcal{L})^{\alpha / 2} u(t, \cdot)\right\|_{L^{2}(G)}^2 \nonumber &=\sum_{[\xi] \in \widehat{G}} d_{\xi}  \left\|\sigma_{(-\mathcal{L})^{\alpha / 2}}(\xi)\widehat{u}(t, \xi) \right\|_{HS}^{2} \\\nonumber=&\sum_{[\xi] \in \widehat{G}} d_{\xi} \sum_{k, \ell=1}^{d_{\xi}}\lambda_\xi^{2\alpha }\left|\widehat{u}(t, \xi)_{k \ell}\right|^{2}\\\nonumber
		=&\sum_{[\xi] \in \mathcal{R}_1} d_{\xi} \sum_{k, \ell=1}^{d_{\xi}}\lambda_\xi^{2\alpha } \left|\widehat{u}(t, \xi)_{k \ell}\right|^{2}  +\sum_{[\xi] \in \mathcal{R}_2} d_{\xi} \sum_{k, \ell=1}^{d_{\xi}}\lambda_\xi^{2\alpha }\left|\widehat{u}(t, \xi)_{k \ell}\right|^{2}  \\		\nonumber
		\lesssim&(1+t)^{-1}\sum_{[\xi] \in \mathcal{R}_1} d_{\xi} \sum_{k, \ell=1}^{d_{\xi}}   \left(\left|\widehat{u}_0(\xi)_{k \ell}\right|^2+\left|\widehat{u}_1(\xi)_{k \ell}\right|^2\right)\\ \nonumber\quad &+{e}^{-c t}\sum_{[\xi] \in \mathcal{R}_2} d_{\xi} \sum_{k, \ell=1}^{d_{\xi}}  \left(\lambda_{\xi}^{2\alpha}\left|\widehat{u}_0(\xi)_{k \ell}\right|^2+\left|\widehat{u}_1(\xi)_{k \ell}\right|^2\right)
		\\			\lesssim &(1+t)^{-1} (\left\|u_{0}\right\|_{H_{\mathcal{L}}^{{\alpha }}(G)}^{2}+\left\|u_{1}\right\|_{L^{2}(G)}^{2}). \end{align}

	%
	%
	%
	
	\textbf{Estimate for $\left\|\partial_{t} u(t, \cdot )\right\|_{L^{2}(G)}$:} From (\ref{number27}) and  (\ref{number26}),     
	the Plancherel formula yields that 
	\begin{align}\label{deri}\nonumber
		\left\|\partial_{t} u(t, \cdot )\right\|_{L^{2}(G)}^2 &=\sum_{[\xi] \in \widehat{G}} d_{\xi} \sum_{k, \ell=1}^{d_{\xi}}  \left|\partial_{t} \widehat{u}(t, \xi)_{k \ell}\right|^{2}\\\nonumber	=&\sum_{[\xi] \in \mathcal{R}_1} d_{\xi} \sum_{k, \ell=1}^{d_{\xi}} \left|\partial_{t} \widehat{u}(t, \xi)_{k \ell}\right|^{2} +\sum_{[\xi] \in \mathcal{R}_2} d_{\xi} \sum_{k, \ell=1}^{d_{\xi}} \left|\partial_{t} \widehat{u}(t, \xi)_{k \ell}\right|^{2} \\		\nonumber	\lesssim&(1+t)^{-2}\sum_{[\xi] \in \mathcal{R}_1} d_{\xi} \sum_{k, \ell=1}^{d_{\xi}}   \left(\left|\widehat{u}_0(\xi)_{k \ell}\right|^2+\left|\widehat{u}_1(\xi)_{k \ell}\right|^2\right)\\ \nonumber\quad &+{e}^{-2c t}\sum_{[\xi] \in \mathcal{R}_2} d_{\xi} \sum_{k, \ell=1}^{d_{\xi}}  \left(\lambda_{\xi}^{2\alpha}\left|\widehat{u}_0(\xi)_{k \ell}\right|^2+\left|\widehat{u}_1(\xi)_{k \ell}\right|^2\right)
		\\			\lesssim &(1+t)^{-2} (\left\|u_{0}\right\|_{H_{\mathcal{L}}^{{\alpha }}(G)}^{2}+\left\|u_{1}\right\|_{L^{2}(G)}^{2}). \end{align}

	Now, we are in a position to prove  Proposition  \ref{thm11}.

	\begin{proof}[Proof of Proposition \ref{thm11}]
		The proof of Theorem \ref{thm11} follows from  the estimates  (\ref{L2}), (\ref{f1}), and  (\ref{deri})  for $\|u(t, \cdot )\|_{L^{2}(G)}$, $\left\|(-\mathcal{L})^{\alpha / 2} u(t, \cdot )\right\|_{L^{2}(G)}$, and 	$\left\|\partial_{t} u(t, \cdot )\right\|_{L^{2}(G)}$, respectively.
	\end{proof}

	\subsection{Local in time existence}\label{sec4}
	In this subsection we will prove  Theorem \ref{thm22}, i.e.,   the local well-posedness of the   Cauchy problem   (\ref{eq0010})  in the energy evolution space  $\mathcal C\left([0,T],  H^\alpha_{\mathcal{L}}(G)\right)\cap\mathcal C^1\left([0,T],L^2(G)\right)$.  

	%
	First, we recall some notations to present the proof of Theorem \ref{thm22}. Consider the space \[X(T):=\mathcal{C}\left([0,T],  H^\alpha_{\mathcal L}(G)\right)\cap\mathcal C^1\left([0,T],L^2(G)\right),\] equipped with the norm
	\begin{align}\label{eq33333} 
		\|u\|_{X(T)}&:=\sup\limits_{t\in[0,T]}\left (  \|u(t,\cdot)\|_{L^2(G)}+\|(-\mathcal L)^{\alpha/2}u(t,\cdot)\|_{L^2(G)}+\|\partial_tu(t,\cdot)\|_{L^2(G)}\right ).
	\end{align}

	Here we will briefly recall the notion of mild solutions in our framework to the Cauchy problem (\ref{eq0010})  and will analyze our approach to prove Theorem \ref{thm22}. Applying Duhamel's principle, the solution to the nonlinear inhomogeneous problem
	\begin{align}\label{eq3111}
		\begin{cases}
			\partial^2_tu+(-\mathcal{L})^\alpha u+\partial_{t}u =F(t, x), & x\in G,t>0,\\
			u(0,x)=  u_0(x),  & x\in G,\\ \partial_tu(0, x)=  u_1(x), & x\in G,
		\end{cases}
	\end{align}
	can be expressed as
	$$ u(t, x)= u_{0}(x)*_{(x)}E_{0}(t, x)+u_{1}(x)*_{(x)}E_{1}(t, x) +\int_{0}^{t} F(s, x)*_{(x)} E_{1}(t-s, x) \;d s,  $$
	where $*_{(x)}$ is  the group convolution product on $G$ with respect to the $x$ variable. Here $E_{0}(t, x)$ and $E_{1}(t, x)$  represent   the fundamental solutions to the homogeneous problem, i.e.,  \eqref{eq3111} with $F=0$ and the initial data $\left(u_{0}, u_{1}\right)=\left(\delta_{0}, 0\right)$ and $\left(u_{0}, u_{1}\right)=$ $\left(0, \delta_{0}\right)$, respectively. 
	For any left-invariant differential operator $L$ on the compact Lie group $ {G}$, we apply  the property  that it commute with the group convolution, that is,   $L\left(v*_{(x)} E_{1}(t, \cdot)\right)=v *_{(x)} L\left(E_{1}(t, \cdot)\right)$ and   the invariance by time translations for the wave operator $	\partial^2_t+(-\mathcal{L})^\alpha $,
	to get the previous representation formula.

	We say that a function  $u$ is  a {\it mild solution} to (\ref{eq3111})  on $[0, T]$ if $u$ is a fixed point for  the       integral operator,  $N: u \in X(T) \rightarrow N u(t, x) ,$ given  by 
	\begin{align}\label{f2} 
		N u(t, x):= \varepsilon u_{0}(x) *_{(x)}  E_{0}(t, x)+\varepsilon u_{1}(x) *_{(x)}  E_{1}(t, x) +\int_{0}^{t}|u(s, x)|^{p} *_{(x)}  E_{1}(t-s, x) \;ds
	\end{align}
	in the energy evolution space $X(T) \doteq \mathcal{C}\left([0, T], H_{\mathcal{L}}^{\alpha}(G)\right) \cap \mathcal{C}^{1}\left([0, T], L^{2}(G)\right)$, equipped with the norm defined in  \eqref{eq33333}. 
	
	In order to show a uniquely determined fixed point of $N$  for a sufficiently small $T=T(\varepsilon)$, we use the Banach fixed point theorem with respect to the norm on $X(T)$ as defined by \eqref{eq33333}.  In fact, for the small enough initial data $\left\|\left(u_{0}, u_{1}\right)\right\|_{H_{\mathcal{L}}^{\alpha}(G) \times L^{2}(G)}$,    we   will establish the following two  inequalities
	
	\begin{align}\label{2number100}
		\|N u\|_{X(T)} \leq C\left\|\left(u_{0}, u_{1}\right)\right\|_{H_{\mathcal{L}}^{\alpha}(G) \times L^{2}(G)}+C\|u\|_{X(T)}^{p},
	\end{align} 
	and \begin{align}\label{2number101}
		\|N u-N v\|_{X(T)} \leq C\|u-v\|_{X(T)}\left(\|u\|_{X(T)}^{p-1}+\|v\|_{X(T)}^{p-1}\right),
	\end{align}
	for any $u, v \in X(T)$ and for some  suitable constant $C>0$ independent of $T$. Then the  Banach fixed point theorem immediately gives  a uniquely determined fixed point $u$ on $N$. This fixed point $u$ will be  our mild solution to (\ref{eq3111})  on $[0, T]$.

	In order to prove the local existence result, an essential tool is the following Gagliardo-Nirenberg type inequality proved in general Lie groups   \cite{Gall}. 
	\begin{lem}\cite{Gall} \label{lemma1}
		Let $G$ be a connected unimodular Lie group with topological dimension $n.$ For any $1<q_0<\infty,~0<q,q_1<\infty$ and $0<\alpha<n$ such that $q_0<\frac{n}{\alpha},$ the following Gagliardo-Nirenberg type inequality holds
		\begin{align}\label{eq33}
			\|f\|_{L^q(G)}\lesssim \|f\|^\theta_{H^{\alpha,q_0}_\mathcal L(G)}\|f\|^{1-\theta}_{L^{q_1}(G)},
		\end{align} 
		for all $f\in H^{\alpha,q_0}_\mathcal L(G)\cap L^{q_1}(G),$ provided that
		\begin{align*}
			\theta=\theta(n,\alpha,q,q_0,q_1)=\frac{\frac{1}{q_1}-\frac{1}{q}}{\frac{1}{q_1}-\frac{1}{q_0}+\frac{\alpha}{n}}\in[0,1].
		\end{align*}
	\end{lem} 
	We refer to \cite{Gall, 27}  for several immediate important  remarks from Lemma \ref{lemma1}.
		%
		The next corollary is a  version of Lemma \ref{lemma1}, which is useful in our setting.
		\begin{cor}
			Let $G$ be a connected unimodular Lie group with topological dimension $n\geq 2[\alpha]+2$.  For any $q\ge2$ such that $q\leq\frac{2n}{n-2\alpha}$, the following Gagliardo-Nirenberg type inequality holds
			\begin{align}\label{eq34}
				\|f\|_{L^q(G)}\lesssim \|f\|^{\theta(n, q, \alpha)}_{H^{\alpha }_\mathcal L(G)}\|f\|^{1-\theta(n, q, \alpha)}_{L^{2}(G)},
			\end{align} 
			for all $f\in H^{\alpha}_\mathcal L(G)$, where $\theta(n, q, \alpha)=\frac{n}{\alpha}\left(\frac{1}{2}-\frac{1}{q} \right) $.
		\end{cor}
		
		\begin{proof}[Proof of Theorem \ref{thm22}]
			The  expression    (\ref{f2})  can be wriiten as  $N u=u^\sharp+I[u]$, where 
			\begin{align*}
				u^\sharp(t,x)=\varepsilon u_{0}(x) *_{(x)}  E_{0}(t, x)+\varepsilon u_{1}(x) *_{(x)}  E_{1}(t, x),
			\end{align*}
			and 
			\begin{align*}
				I[u](t,x):=\int\limits_0^t |u(s,x)|^p*_x E_1(t-s, x)ds.
			\end{align*} 
			
			Now, for the part $u^\sharp$,   Theorem \ref{thm11},   immediately  implies that
			\begin{align}\label{f3}
				\|u^\sharp\|_{X(T)}\lesssim\varepsilon\|(u_0,u_1)\|_{{H}_{\mathcal L}^\alpha (G)\times L^2(G)}.
			\end{align}
			On the other hand, for the part $I[u]$, using Minkowski's integral inequality, Young's convolution inequality, Theorem \ref{thm11}, and  by time translation invariance property of the Cauchy problem (\ref{eq0010}), we get  
			\begin{align}\label{f}\nonumber
				\|\partial_t^j(-\mathcal L)^{i\alpha/2}I[u]\|_{L^2(G)}&=\left(\int_{G}	\big |\partial_t^j(-\mathcal L)^{i\alpha/2} \int\limits_0^t |u(s,x)|^p*_x E_1(t-s, x)ds\big |^2 dg\right)^{\frac{1}{2}}\\\nonumber
				&=\left(\int_{G}\big	|  \int\limits_0^t |u(s,x)|^p*_x \partial_t^j(-\mathcal L)^{i\alpha/2}E_1(t-s, x)ds\big|^2 dg\right)^{\frac{1}{2}}	\\\nonumber
				&\lesssim  \int\limits_0^t \| |u(s,\cdot )|^p*_x \partial_t^j(-\mathcal L)^{i\alpha/2}E_1(t-s, \cdot)\|_{L^2(G)}ds\\\nonumber
				&\lesssim  \int\limits_0^t \| u(s,\cdot)^p\|_{L^2(G)} \|\partial_t^j(-\mathcal L)^{i\alpha/2}E_1(t-s, \cdot)\|_{L^2(G)}ds\\\nonumber
				&\lesssim \int\limits_0^t (1+t-s)^{-j-\frac{i}{2}}\|u(s,\cdot)\|^p_{L^{2p}(G)}ds\\\nonumber
				&\lesssim\int\limits_0^t (1+t-s)^{-j-\frac{i}{2}} \|u(s,\cdot)\|^{p\theta(n,2p, \alpha)}_{H^\alpha_\mathcal L(G)}\|u(s,\cdot)\|^{p(1-\theta(n,2p,\alpha ))}_{L^2(G)}ds \\ 
				&\lesssim\int\limits_0^t (1+t-s)^{-j-\frac{i}{2}}  \|u\|^p_{X(s)}ds \lesssim t    \|u\|^p_{X(t)},
			\end{align} 
			for $i,j\in\{0,1\}$, such that $0\leq i+j\leq 1.$   Again for $i,j\in\{0,1\},$ such that $0\leq i+j\leq 1,$ a similar calculations as in  (\ref{f}) together with H\"older's inequality  and    (\ref{eq34}),  we get 
			\begin{align}\label{f5}\nonumber
				& 	\|\partial_t^j(-\mathcal L)^{i\alpha/2}\left(I[u]-I[v]\right)\|_{L^2(G)}\\\nonumber&\lesssim \int\limits_0^t (1+t-s)^{-j-\frac{i}{2}}\||u(s,\cdot)|^p-|v(s,\cdot)|^p\|_{L^{2}(G)}ds\\\nonumber
				&\lesssim\int\limits_0^t (1+t-s)^{-j-\frac{i}{2}} \|u(s,\cdot)-v(s,\cdot)\|_{L^{2p}(G)}\left(\|u(s,\cdot)\|^{p-1}_{L^{2p}(G)}+\|v(s,\cdot)\|^{p-1}_{L^{2p}(G)}\right)ds\\ &\lesssim t  \|u-v\|_{X(t)}\left(\|u\|^{p-1}_{X(t)}-\|v\|^{p-1}_{X(t)}\right).
			\end{align} 
			Thus 	combining  (\ref{f3}),   (\ref{f}), and (\ref{f5}), we have
			\begin{align}\label{1}
				\|N u\|_{X(T)} \leq D \varepsilon\left\|\left(u_{0}, u_{1}\right)\right\|_{H_{\mathcal{L}}^{\alpha }(G) \times L^{2}(G)}+DT\|u\|_{X(t)}^{p} 
			\end{align} 
			and 
			\begin{align}\label{2} \|Nu-Nv\|_{X(T)}\leq   DT \|u-v\|_{X(t)}\left(\|u\|^{p-1}_{X(T)}-\|v\|^{p-1}_{X(T)}\right),\end{align} 
			where $D$ is a constant independent of $t$. 
			Choose $T$ (sufficiently small) in such a way that the map $N$ turns out to be a contraction in some neighborhood of $0$ in the Banach space $X(T).$ Therefore, Banach's fixed point theorem gives us the uniquely determined fixed point $u$ for the map $N$, which is our mild solution. 
			This completes the proof.
		\end{proof}
		
		From the above local existence result, we have the following remark.
		\begin{rmk}
			We note that in the statement of Theorem \ref{thm22}, the  restriction on the upper bound for the exponent $p$ which is $p\leq\frac{n}{n-2\alpha }$   is necessary in order to apply Gagliardo-Nirenberg type inequality (\ref{eq34}) in  (\ref{f}). Also, the other restriction $n \geq 2[\alpha]+2$ is made to fulfill the assumptions for the employment of such inequality.
		\end{rmk}		
		
		%
			%
		%
		%
		
		\subsection{Blow-up result}\label{sec5}
		In this subsection,  we prove Theorem \ref{f6}   using a comparison argument for ordinary differential inequality of second order. 
		Now we are ready to prove our main result of this section using an iteration argument.
		\begin{proof}[Proof of Theorem \ref{f6}]
			According to Definition \ref{eq332},  let $u$ be a local in-time energy solution to (\ref{eq0010})  with lifespan $T$. Let    $t \in(0, T)$ be fixed.  Suppose that $\phi \in \mathcal{C}_{0}^{\infty}([0, T) \times G),$  is a cut-off function such that $\phi=1$ on $[0, t] \times G$ in (\ref{eq011}). Then 
			\begin{align}\label{f7}
				\int_{G} \partial_{t} u(t, x) \;dx+	\int_{G}   u(t, x) \;dx-\varepsilon \int_{G} u_{0}(x) \;dx-\varepsilon \int_{G} u_{1}(x) \;dx=\int_{0}^{t} \int_{ {G}}|u(s, x)|^{p}  {~d} x {~d} s
			\end{align}
			Let us introduce the time-dependent functional
			$$
			U_{0}(t) \doteq \int_{G} u(t, x) \;dx. 
			$$
			Then the equality (\ref{f7})    can be rewritten in the following way:
			$$
			U_{0}^{\prime}(t)-U_{0}^{\prime}(0)+	U_{0}(t)-U_{0}(0) =\int_{0}^{t} \int_{G}|u(s, x)|^{p} \;dx \;ds .
			$$
			We also remark that,  from the assumptions on the initial data, we obtain 
			$$
			U_{0}(0)=\varepsilon \int_{G} u_{0}(x) \;dx \geq 0 \quad \text { and } \quad U_{0}^{\prime}(0)=\varepsilon \int_{G} u_{1}(x) \;dx \geq 0.
			$$
			Using Jensen's inequality, we have 
			\begin{align}\label{number30}
				U_{0}^{\prime}(t)-U_{0}^{\prime}(0)+	U_{0}(t)-U_{0}(0)  \geq \int_{0}^{t}  \left|U_{0}(s)\right|^{p} \;ds.
			\end{align}
			Multiplying both sides of \eqref{number30}  by $e^t$ and then  integrating over $[0, t ]$,   we obtain
			$$
			{e}^t U_0(t) \geq  \left(U_0^{\prime}(0)+U_0(0)\right)\left( {e}^t-1\right)+U_0(0)+\int_0^t  {e}^\eta \int_0^\eta\left|U_0(s)\right|^p  {~d} s  {~d} \eta,
			$$
			i.e.,
			$$
			U_0(t) \geq U_0(0)+U_0^{\prime}(0)\left(1- {e}^{-t}\right)+\int_0^t  {e}^{\eta-t} \int_0^\eta\left|U_0(s)\right|^p  {~d} s  {~d} \eta.
			$$
			Since $U_0(0)$ and $U_0'(0)$ are non-negative, the above expression implies that  $U_0$ is a positive function. Moreover, we also can say that
			$$
			U_0(t) \geq U_0(0)+U_0^{\prime}(0)\left(1- {e}^{-t}\right) \geq C \varepsilon \quad \text { for } t \ge0,
			$$
			where the multiplicative constant $C$ depends on $u_0, u_1$ and  we also have the    following  iteration scheme
			$$
			U_0(t) \geq \int_0^t  {e}^{\eta-t} \int_0^\eta\left|U_0(s)\right|^p  {~d} s  {~d} \eta.
			$$
			Now proceeding similarly as in Subsection 3.1 and 3.2 of \cite{27} for the iteration argument,   we conclude the proof of Theorem \ref{f6}.
		\end{proof}

			%
			%
		%
		%
		
		\section{A global existence result}\label{sec6}
		In this section,   we study the global in-time existence of small data solutions for the nonlinear fractional dumped wave equation with mass and the power type nonlinearity. More preciously,  for $0<\alpha<1$, we consider the   Cauchy problem  
		\begin{align} \label{2number31}
			\begin{cases}
				\partial^2_tu+\left( -\mathcal{L}\right)^\alpha u+b\partial_{t} u+m^2u=|u|^p, & x\in  G, t>0,\\
				u(0,x)=u_0(x),  & x\in G,\\ \partial_tu(x,0)=u_1(x), & x\in G,
			\end{cases}
		\end{align}
		where  $ p > 1$, $b,m^2$ are positive constants, $u_{0}(x)$ and $u_{1}(x)$ are two given functions on $G$.

		\subsection{Fourier multiplier expressions and $L^2(G)-L^2(G)$ estimates 
		}\label{2sec3}
		In this subsection, we derive $L^2(G)– L^2(G)$ estimates for  solutions of  the homogeneous problem     (\ref{2number31}).
		We  employ the group Fourier transform on the compact group $G$ with respect to  the space variable $x$  together with the Plancherel identity in order to      estimate     $L^2$-norms of  $u(t, ·), (-\mathcal{L})^{\frac{\alpha}{2}}u(t, \cdot)$, and $\partial_{t}u(t, ·)$. 
		
		Let $u$ be a solution to (\ref{2number31}). Let $\widehat{u}(t, \xi)=(\widehat{u}(t, \xi)_{kl})_{1\leq k, l\leq d_\xi}\in \mathbb{C}^{d_\xi\times d_\xi}, [\xi]\in\widehat{ G}$ denote the Fourier transform of $u$  with respect to the $x $ variable. Applying the group Fourier transform with respect to $x$ on   (\ref{2number31}), we deduce that $\widehat{u}(t, \xi)$ is  a solution to the following  Cauchy problem for the system of ODE's (with the size of the system that depends on the representation $\xi$)
		\begin{align}\label{2eq6661}
			\begin{cases}
				\partial^2_t\widehat{u}(t,\xi)+(-	\sigma_{\mathcal{L}}(\xi))^\alpha \widehat{u}(t,\xi) +b \partial_t \widehat{u}(t,\xi)+m^2  \widehat{u}(t,\xi)=0,& [\xi]\in\widehat{ G},~t>0,\\ \widehat{u}(0,\xi)=\widehat{u}_0(\xi), &[\xi]\in\widehat{ G},\\ \partial_t\widehat{u}(0,\xi)=\widehat{u}_1(\xi), &[\xi]\in\widehat{ G},
			\end{cases} 
		\end{align}
		where  $\sigma_{\mathcal{L}}$	is the symbol of  the  Laplace-Beltrami operator $\mathcal{L}$.  Using the identity (\ref{symbol}),  the   system  (\ref{2eq6661}) can be written in the form of $d_\xi^2$ independent   ODE's, namely,
		\begin{align}\label{2eqq7}
			\begin{cases}
				\partial^2_t\widehat{u}(t,\xi)_{kl}+b	\partial_t\widehat{u}(t,\xi)_{kl}+ \lambda_\xi^{2\alpha }  \widehat{u}(t,\xi)_{kl}+m^2\widehat{u}(t,\xi)_{kl}= 0,& [\xi]\in\widehat{ G},~t>0,\\ \widehat{u}(0,\xi)_{kl}=\widehat{u}_0(\xi)_{kl}, &[\xi]\in\widehat{ G},\\ \partial_t\widehat{u}(0,\xi)_{kl}=\widehat{u}_1(\xi)_{kl}, &[\xi]\in\widehat{ G},
			\end{cases}
		\end{align}
		for all $k,l\in\{1,2,\ldots,d_\xi\}.$ 	Then, the characteristic equation of (\ref{2eqq7}) is given by
		\[\lambda^2+b	\lambda+\lambda_\xi^{2\alpha }+m^2 =0,\]
		and consequently the characteristic roots are   $\lambda=-\frac{b}{2}\pm  {\sqrt{\frac{b^2}{4}-\lambda_\xi^{2\alpha}-m^2}} $. 
		Thus the solution to the   homogeneous  problem   (\ref{2eqq7}) is given by 
		\begin{align}\label{2number2}
			\widehat{u}(t,\xi)_{kl}=e^{-\frac{bt}{2}}A_0(t, \xi) \widehat{u}_0(\xi)_{kl}+e^{-\frac{bt}{2}}A_1(t, \xi) \left(\widehat{u}_1(\xi)_{kl}+\frac{b}{2}\widehat{u}_0(\xi)_{kl}\right),
		\end{align}
		where
		\begin{align}\label{2number1}
			A_0(t, \xi)=\begin{cases}
				\cosh \left(\sqrt{\frac{b^2}{4}-\lambda_\xi^{2\alpha}-m^2}~~t \right),& \text{if } \lambda_\xi^{2\alpha}< \frac{b^2}{4}-m^2,\\
				1,& \text{if } \lambda_\xi^{2\alpha}= \frac{b^2}{4}-m^2,\\
				\cos\left(\sqrt{\lambda_\xi^{2\alpha}-\frac{b^2}{4}+m^2}~~t \right),& \text{if } \lambda_\xi^{2\alpha}> \frac{b^2}{4}-m^2,\\
		\end{cases} 	\end{align}
		and
		\begin{align}\label{2number3}
			A_1(t, \xi)=\begin{cases}
				\frac{	2\sinh \left(\sqrt{\frac{b^2}{4}-\lambda_\xi^{2\alpha}-m^2}~~t \right)}{\sqrt{\frac{b^2}{4}-\lambda_\xi^{2\alpha}-m^2}},& \text{if } \lambda_\xi^{2\alpha}< \frac{b^2}{4}-m^2,\\
				t,& \text{if }  \text{if } \lambda_\xi^{2\alpha}= \frac{b^2}{4}-m^2,\\
				\frac{	\sin\left(\sqrt{\lambda_\xi^{2\alpha}+\frac{b^2}{4}-m^2}~~t \right)}{\sqrt{\lambda_\xi^{2\alpha}-\frac{b^2}{4}+m^2 }},& \text{if } \lambda_\xi^{2\alpha}> \frac{b^2}{4}-m^2.
			\end{cases} 
		\end{align}
		We 	notice that $A_0(t, \xi) = \partial_t A_1(t, \xi)$ for any $[\xi] \in \widehat{ G}$. Moreover, we have the following representation for the time derivative
		\begin{align}\label{2number22}
			\partial_{t}	\widehat{u}(t,\xi)_{kl}=e^{-\frac{bt}{2}}A_0(t, \xi)  \widehat{u}_1(\xi)_{kl}-e^{-\frac{bt}{2}} A_1(t, \xi)  \left[  \frac{b}{2}\widehat{u}_1(\xi)_{kl}+ (\lambda_\xi^{2\alpha}+m^2)\widehat{u}_0(\xi)_{kl} \right]  .
		\end{align}
		
		Next we will estimate the values of $	 |\widehat{u}(t, \xi)_{k \ell}|, \partial_t	 |\widehat{u}(t, \xi)_{k \ell}|$ and $	\lambda_{\xi}| \widehat{u}(t, \xi)_{k \ell}   |$ by     considering    the   relation between $b$ and $m^2$.

		\textbf{When  $b^2<4 m^2$: }
		The only the case is to consider that  $\lambda_{\xi}^2>\frac{b^2}{4}-m^2$ by considering the fact  that  all eigenvalues $\{\lambda_{\xi}^{2\alpha}\}_{[\xi] \in \widehat{ G}}$ of $(-\mathcal{L})^{\alpha}$ are nonnegative. Thus, by the similar calculus done in Subsection \ref{sec3.1},  we have
		\begin{align}\label{2number32}
			\left|\widehat{u}(t, \xi)_{k \ell}\right| & \lesssim  {e}^{-\frac{b}{2} t}\left[ \left|\widehat{u}_0(\xi)_{k \ell}\right|+\left|\widehat{u}_1(\xi)_{k \ell}\right|\right], 	\end{align}
		\begin{align}\label{2number33}
			\lambda_{\xi}^\alpha\left|\widehat{u}(t, \xi)_{k \ell}\right| & \lesssim  {e}^{-\frac{b}{2} t}\left[ \left(1+\lambda_{\xi}^\alpha\right)\left|\widehat{u}_0(\xi)_{k \ell}\right|+\left|\widehat{u}_1(\xi)_{k \ell}\right|\right],
		\end{align}
		and 
		\begin{align}\label{2number34}\nonumber
			\left|\partial_t \widehat{u}(t, \xi)_{k \ell}\right|& \lesssim	e^{-\frac{bt}{2}}  \widehat{u}_1(\xi)_{kl} +   A_0(t, \xi)   (\lambda_\xi^{2\alpha}+m^2)\widehat{u}_0(\xi)_{kl}  \\& \lesssim  {e}^{-\frac{b}{2} t}\left[\left(1+\lambda_{\xi}^\alpha \right)\left|\widehat{u}_0(\xi)_{k \ell}\right|+\left|\widehat{u}_1(\xi)_{k \ell}\right|\right],
		\end{align}
		for any $t \geq0$. Thus,	using the Plancherel formula along with  the equations (\ref{2number32}), (\ref{2number33}) and (\ref{2number34}), it follows that 
		\begin{align}\label{2L2}\nonumber
			\|\partial_t^j(-\mathcal L)^{i\alpha/2}u(t, \cdot)\|_{L^{2}(G)}^{2}&=\sum_{[\xi] \in \widehat{G}} d_{\xi} \sum_{k, \ell=1}^{d_{\xi}}\lambda_{\xi}^{2\alpha  i}\left|\partial_t^j \widehat{u}(t, \xi)_{k \ell}\right|^2    \\\nonumber
			& \lesssim  {e}^{-b t} \sum_{[\xi] \in \widehat{ G}} d_{\xi} \sum_{k, \ell=1}^{d_{\xi}}\left(\left(1+\lambda_{\xi}^{2\alpha}\right)^{(i+j)}\left|\widehat{u}_0(\xi)_{k \ell}\right|^2+\left|\widehat{u}_1(\xi)_{k \ell}\right|^2\right) \\ 
			& =  {e}^{-b t}\left[  \left\|u_{0}\right\|_{H_{\mathcal{L}}^{{\alpha(i+j) }}(G)}^{2}+\left\|u_{1}\right\|_{L^{2}(G)}^{2} \right],
		\end{align}
		for any 	  $i,j\in\{0,1\}$, such that $0\leq i+j\leq 1$,  with the convention that $H_{\mathcal{L}}^{0}(G)=L^2(G).$

		\textbf{When  $b^2=4 m^2$: } In this case,  we only   have to consider  the cases when $\lambda_{\xi}^{2\alpha}=0$ and  $\lambda_{\xi}^{2\alpha}>0$. Then, from \eqref{2number2}, \eqref{2number1}, and \eqref{2number3}, the solution can be written as 
		\begin{align*}
			\widehat{u}(t, \xi)_{k \ell}=
			\begin{cases}
				e^{-\frac{bt}{2}}	\cos\left( {\lambda_\xi^{\alpha}}~~t \right) \widehat{u}_0(\xi)_{kl}+e^{-\frac{bt}{2}}\frac{	\sin\left( {\lambda_\xi^{\alpha} }~~t \right)}{ {\lambda_\xi^{\alpha}  }} \left(\widehat{u}_1(\xi)_{kl}+\frac{b}{2}\widehat{u}_0(\xi)_{kl}\right),&   \text {if} ~ \lambda_{\xi}^2>0, \\ 
				e^{-\frac{bt}{2}}	  \widehat{u}_0(\xi)_{kl}+te^{-\frac{bt}{2}}  \left(\widehat{u}_1(\xi)_{kl}+\frac{b}{2}\widehat{u}_0(\xi)_{kl}\right),&   \text {if} ~ \lambda_{\xi}^2=0.
		\end{cases}  	\end{align*}	
		The second case $\lambda_{\xi}^2=0$ needs to be included as $0$ is the eigenvalue for the trivial representation $G.$   Thus 
		$$
		\begin{aligned}
			\left|\widehat{u}(t, \xi)_{k \ell}\right| & \lesssim (1+t) {e}^{-\frac{b}{2} t}\left(\left|\widehat{u}_0(\xi)_{k \ell}\right|+\left|\widehat{u}_1(\xi)_{k \ell}\right|\right) ,\\
			\lambda_{\xi}^\alpha\left|\widehat{u}(t, \xi)_{k \ell}\right| & \lesssim  {e}^{-\frac{b}{2} t}\left[ \left(1+\lambda_{\xi}^\alpha\right)\left|\widehat{u}_0(\xi)_{k \ell}\right|+\left|\widehat{u}_1(\xi)_{k \ell}\right|\right],
		\end{aligned}
		$$
		and 
		\begin{align*} 
			\left|\partial_t \widehat{u}(t, \xi)_{k \ell}\right| \lesssim  (1+t) {e}^{-\frac{b}{2} t}\left[\left(1+\lambda_{\xi}^\alpha \right)\left|\widehat{u}_0(\xi)_{k \ell}\right|+\left|\widehat{u}_1(\xi)_{k \ell}\right|\right],
		\end{align*}
		for any $t \geq0$. Thus	using the Plancherel formula along with  the above estimates, we get 
		\begin{align}\label{2L22}\nonumber
			\|\partial_t^j(-\mathcal L)^{i\alpha/2}u(t, \cdot)\|_{L^{2}(G)}^{2}&=\sum_{[\xi] \in \widehat{G}} d_{\xi} \sum_{k, \ell=1}^{d_{\xi}}\lambda_{\xi}^{2\alpha  i}\left|\partial_t^j \widehat{u}(t, \xi)_{k \ell}\right|^2    \\\nonumber
			& \lesssim (1+t)^2 {e}^{-b t} \sum_{[\xi] \in \widehat{ G}} d_{\xi} \sum_{k, \ell=1}^{d_{\xi}}\left(\left(1+\lambda_{\xi}^{2\alpha}\right)^{(i+j)}\left|\widehat{u}_0(\xi)_{k \ell}\right|^2+\left|\widehat{u}_1(\xi)_{k \ell}\right|^2\right) \\ 
			& =  (1+t)^2 {e}^{-b t}\left[  \left\|u_{0}\right\|_{H_{\mathcal{L}}^{{\alpha(i+j) }}(G)}^{2}+\left\|u_{1}\right\|_{L^{2}(G)}^{2} \right],
		\end{align}
		for any 	  $i,j\in\{0,1\}$, such that $0\leq i+j\leq 1$.

		
		\textbf{When  $b^2>4 m^2$: }	In this case, depending on the range of  $\lambda_{\xi}^2$,  the characteristic roots may be complex conjugate or real distinct, or they may coincide.   But comparing all possible cases in   \eqref{2number1} and \eqref{2number3} and keeping in mind that   the regularity is provided from the case with complex conjugate characteristic roots, whereas the decay rate is given by the continuous irreducible unitary representations with $\lambda_{\xi}^2=0$,  we obtain 
		$$
		\begin{aligned}
			\left|\widehat{u}(t, \xi)_{k \ell}\right| & \lesssim  {e}^{\left(-\frac{b}{2}+\sqrt{ \frac{b^2}{4}-m^2}\right) t}\left(\left|\widehat{u}_0(\xi)_{k \ell}\right|+\left|\widehat{u}_1(\xi)_{k \ell}\right|\right) ,\\
			\lambda_{\xi}^\alpha\left|\widehat{u}(t, \xi)_{k \ell}\right| & \lesssim   {e}^{\left(-\frac{b}{2}+\sqrt{ \frac{b^2}{4}-m^2}\right) t} \left[ \left(1+\lambda_{\xi}^\alpha\right)\left|\widehat{u}_0(\xi)_{k \ell}\right|+\left|\widehat{u}_1(\xi)_{k \ell}\right|\right],
		\end{aligned}
		$$
		and 
		\begin{align*} 
			\left|\partial_t \widehat{u}(t, \xi)_{k \ell}\right| \lesssim    {e}^{\left(-\frac{b}{2}+\sqrt{ \frac{b^2}{4}-m^2}\right) t} \left[\left(1+\lambda_{\xi}^\alpha \right)\left|\widehat{u}_0(\xi)_{k \ell}\right|+\left|\widehat{u}_1(\xi)_{k \ell}\right|\right],
		\end{align*}
		for any $t \geqslant 0$.
		Thus	using the Plancherel formula along with  the above estimates, we get 
		\begin{align}\label{2L222}\nonumber
			\|\partial_t^j(-\mathcal L)^{i\alpha/2}u(t, \cdot)\|_{L^{2}(G)}^{2}&=\sum_{[\xi] \in \widehat{G}} d_{\xi} \sum_{k, \ell=1}^{d_{\xi}}\lambda_{\xi}^{2\alpha  i}\left|\partial_t^j \widehat{u}(t, \xi)_{k \ell}\right|^2    \\\nonumber
			& \lesssim {e}^{\left(- {b}+\sqrt{  {b^2}-4m^2}\right) t} \sum_{[\xi] \in \widehat{ G}} d_{\xi} \sum_{k, \ell=1}^{d_{\xi}}\left(\left(1+\lambda_{\xi}^{2\alpha}\right)^{(i+j)}\left|\widehat{u}_0(\xi)_{k \ell}\right|^2+\left|\widehat{u}_1(\xi)_{k \ell}\right|^2\right) \\ 
			& = {e}^{\left(- {b}+\sqrt{  {b^2}-4m^2}\right) t} \left[  \left\|u_{0}\right\|_{H_{\mathcal{L}}^{{\alpha(i+j) }}(G)}^{2}+\left\|u_{1}\right\|_{L^{2}(G)}^{2} \right],
		\end{align}
		for any 	  $i,j\in\{0,1\},$ such that $0\leq i+j\leq 1$.

		Now, we are in a position to prove  Proposition  \ref{2thm11}.
		\begin{proof}[Proof of Proposition \ref{2thm11}]
			The proof of Proposition \ref{2thm11} follows from  the estimates  (\ref{2L2}), (\ref{2L22}), and  (\ref{2L222})  for $\|u(t, \cdot )\|_{L^{2}(G)}$, $\left\|(-\mathcal{L})^{\alpha / 2} u(t, \cdot )\right\|_{L^{2}(G)}$, and 	$\left\|\partial_{t} u(t, \cdot )\right\|_{L^{2}(G)}$, respectively.
		\end{proof}

		\subsection{Global in time  existence}\label{2sec4}
		This subsection is devoted to  prove  Theorem \ref{2thm22}, i.e.,  the     global existence of small data solutions for the  fractional   Cauchy problem   (\ref{2number31})  in the energy evolution space  $\mathcal C\left([0,T],  H^\alpha_{\mathcal{L}}(G)\right)\cap\mathcal C^1\left([0,T],L^2(G)\right)$.  

		%
		First, we recall some notations to present the proof of Theorem \ref{2thm22}. Consider the space \[X(T):=\mathcal{C}\left([0,T],  H^\alpha_{\mathcal L}(G)\right)\cap\mathcal C^1\left([0,T],L^2(G)\right),\] equipped with the norm
		\begin{align}\label{2eq33333} 
			\|u\|_{X(T)}&:=\sup\limits_{t\in[0,T]} (A_{b, m^2}(t))^{-1} \left (  \|u(t,\cdot)\|_{L^2(G)}+\|(-\mathcal L)^{\alpha/2}u(t,\cdot)\|_{L^2(G)}+\|\partial_tu(t,\cdot)\|_{L^2(G)}\right ),
		\end{align}
		where $A_{b, m^2}(t)$ is given by
		$$
		A_{b, m^2}(t) \doteq\left\{\begin{array}{ll}
			{e}^{-\frac{b}{2} t} & \text { if } b^2<4 m^2, \\
			(t+1)  {e}^{-\frac{b}{2} t} & \text { if } b^2=4 m^2, \\
			{e}^{\left(-\frac{b}{2}+\sqrt{\frac{b^2}{4}-m^2}\right) t} & \text { if } b^2>4 m^2.
		\end{array}\right.
		$$
		
		Here we briefly recall the notion of mild solutions in our framework to the Cauchy problem (\ref{2number31})  and will analyze our approach to prove Theorem \ref{2thm22}. Applying Duhamel's principle, the solution to the nonlinear inhomogeneous problem
		\begin{align}\label{2eq3111}
			\begin{cases}
				\partial^2_tu+(-\mathcal{L})^\alpha u+b\partial_{t}u +m^2u=F(t, x), & x\in G,t>0,\\
				u(0,x)=  u_0(x),  & x\in G,\\ \partial_tu(0, x)=  u_1(x), & x\in G,
			\end{cases}
		\end{align}
		can be expressed as
		$$ u(t, x)= u_{0}(x)*_{(x)}E_{0}(t, x)+u_{1}(x)*_{(x)}E_{1}(t, x) +\int_{0}^{t} F(s, x)*_{(x)} E_{1}(t-s, x) \;d s,  $$
		where $*_{(x)}$ denotes the group convolution product on $G$ with respect to the $x$ variable. Here,  $E_{0}(t, x)$ and $E_{1}(t, x)$  are  the fundamental solutions to the homogeneous problem (\ref{2eq3111}), i.e., when  $F=0$ with initial data $\left(u_{0}, u_{1}\right)=\left(\delta_{0}, 0\right)$ and $\left(u_{0}, u_{1}\right)=$ $\left(0, \delta_{0}\right)$, respectively. 

		For a function  $u$ on   $[0, T]$ to be a mild solution to (\ref{2eq3111}),    we refer to 	  subsection \ref{sec4}. 
		Furthermore,  if the estimates \eqref{2number100} and  \eqref{2number101} hold   uniformly with respect to $T$  then the    solution can be prolonged and defined for any  $t \in  (0,\infty) $ which will be our {\it global solution}. Now we present the proof of Theorem \ref{2thm22}.

		\begin{proof}[Proof of Theorem \ref{2thm22}]
			The  expression    (\ref{f2})  can be wriiten as  $N u=u^\sharp+I[u]$, where 
			\begin{align*}
				u^\sharp(t,x)=\varepsilon u_{0}(x) *_{(x)}  E_{0}(t, x)+\varepsilon u_{1}(x) *_{(x)}  E_{1}(t, x)
			\end{align*}
			and 
			\begin{align*}
				I[u](t,x):=\int\limits_0^t |u(s,x)|^p*_x E_1(t-s, x)ds.
			\end{align*} 
			
			Now, for the part $u^\sharp$,   Theorem \ref{2thm11},   immediately  implies that
			\begin{align}\label{2f3}
				\|u^\sharp\|_{X(T)}\lesssim \|(u_0,u_1)\|_{{H}_{\mathcal L}^\alpha (G)\times L^2(G)}.
			\end{align}
			On the other hand, for the part $I[u]$, using Minkowski's integral inequality, Young's convolution inequality, Theorem \ref{2thm11}, and  by time translation invariance property of the Cauchy problem (\ref{2number31}), we get  
			\begin{align}\label{2f}\nonumber
				\|\partial_t^j(-\mathcal L)^{i\alpha/2}I[u]\|_{L^2(G)}&=\left(\int_{G}	\big |\partial_t^j(-\mathcal L)^{i\alpha/2} \int\limits_0^t |u(s,x)|^p*_x E_1(t-s, x)ds\big |^2 dg\right)^{\frac{1}{2}}\\\nonumber
				&=\left(\int_{G}\big	|  \int\limits_0^t |u(s,x)|^p*_x \partial_t^j(-\mathcal L)^{i\alpha/2}E_1(t-s, x)ds\big|^2 dg\right)^{\frac{1}{2}}	\\\nonumber
				&\lesssim  \int\limits_0^t \| |u(s,\cdot )|^p*_x \partial_t^j(-\mathcal L)^{i\alpha/2}E_1(t-s, \cdot)\|_{L^2(G)}ds\\\nonumber
				&\lesssim  \int\limits_0^t \| u(s,\cdot)^p\|_{L^2(G)} \|\partial_t^j(-\mathcal L)^{i\alpha/2}E_1(t-s, \cdot)\|_{L^2(G)}ds\\\nonumber
				&\lesssim \int\limits_0^t A_{b, m^2}(t-s) \|u(s,\cdot)\|^p_{L^{2p}(G)}ds\\\nonumber
				&\lesssim\int\limits_0^t A_{b, m^2}(t-s) \|u(s,\cdot)\|^{p\theta(n,2p, \alpha)}_{H^\alpha_\mathcal L(G)}\|u(s,\cdot)\|^{p(1-\theta(n,2p,\alpha ))}_{L^2(G)}ds \\ \nonumber
				&\lesssim\int\limits_0^t A_{b, m^2}(t-s) A_{b, m^2}(s)^p \|u\|^p_{X(s)}ds\\
				&\lesssim \|u\|^p_{X(t)} \int\limits_0^t A_{b, m^2}(t-s) A_{b, m^2}(s)^p  ds \leq \|u\|^p_{X(t)} A_{b, m^2}(t),
			\end{align} 
			for $i,j\in\{0,1\}$ such that $0\leq i+j\leq 1.$   Again for $i,j\in\{0,1\}$ such that $0\leq i+j\leq 1,$ a similar calculations as in  (\ref{2f}) together with H\"older's inequality  and    (\ref{eq34}),  we get 
			\begin{align}\label{2f5}\nonumber
				& 	\|\partial_t^j(-\mathcal L)^{i\alpha/2}\left(I[u]-I[v]\right)\|_{L^2(G)}\\\nonumber&\lesssim \int\limits_0^t A_{b, m^2}(t-s) \||u(s,\cdot)|^p-|v(s,\cdot)|^p\|_{L^{2}(G)}ds\\\nonumber
				&\lesssim\int\limits_0^t A_{b, m^2}(t-s) \|u(s,\cdot)-v(s,\cdot)\|_{L^{2p}(G)}\left(\|u(s,\cdot)\|^{p-1}_{L^{2p}(G)}+\|v(s,\cdot)\|^{p-1}_{L^{2p}(G)}\right)ds\\ \nonumber
				&\lesssim  \|u-v\|_{X(t)}\left(\|u\|^{p-1}_{X(t)}-\|v\|^{p-1}_{X(t)}\right) \int\limits_0^t A_{b, m^2}(t-s) A_{b, m^2}(s)^p  ds\\
				&  \leq \|u\|^p_{X(t)} A_{b, m^2}(t).
			\end{align} 
			Thus 	combining  (\ref{2f3}),   (\ref{2f}), and (\ref{2f5}), we have
			\begin{align}\label{21}
				\|N u\|_{X(t)} \leq D  \left\|\left(u_{0}, u_{1}\right)\right\|_{H_{\mathcal{L}}^{\alpha }(G) \times L^{2}(G)}+D\|u\|_{X(t)}^{p} 
			\end{align} 
			and 
			\begin{align}\label{22} \|Nu-Nv\|_{X(T)}\leq   D \|u-v\|_{X(t)}\left(\|u\|^{p-1}_{X(T)}-\|v\|^{p-1}_{X(T)}\right).\end{align} 
			This shows that the map $N$ turns out to be a contraction in some neighborhood of $0$ in the Banach space $X(T).$ Therefore, Banach's fixed point theorem gives us the uniquely determined fixed point $u$  on  $ [0, T ]$ for the map $N$, which is our mild solution.  
			
			Note that, thanks to the exponential decay rate $A_{b,m^2}(t)$ both in (\ref{2f}) and (\ref{2f5}) we have the uniform boundedness of the integral $$(A_{b, m^2}(t))^{-1}\int_0^t A_{b, m^2}(t-s) A_{b, m^2}(s)^p  ds,$$ without    any conditions on $p$.
			This completes the proof of Theorem \ref{2thm22}.
		\end{proof}
		
		We have the following remark regarding Theorem.

		\section{Final remarks}\label{sec7}

		In \cite{shyamm}, we already seen that for the    fractional wave operator  	$\partial^2_t+(-\mathcal{L})^\alpha $ and for  the damped wave operator $	\partial^2_t+(-\mathcal{L})^\alpha +	\partial_t$ defined in Section \ref{sec3}, under some suitable assumptions on the initial data,   the  local in-time solutions to these    Cauchy problem  blow up in finite time for any $p > 1$.   In other words, we do not get any global in-time existence result in this case. However,      in Section \ref{sec6} of this paper,  we have seen that the presence of a positive damping term and a positive mass term in the Cauchy problem completely reverses the scenario. In a similar manner, 	  the fractional damped wave equation on the Heisenberg group   will be considered in a forthcoming paper.

		\section{Data availability statement}
		The authors confirm that the data supporting the findings of this study are available within the article  and its supplementary materials.


\begin{thebibliography}{aaa}
			
			
			\baselineskip=11pt
			
			\bibitem{Ahmad15} B. Ahmad, A. Alsaedi and M. Kirane, Nonexistence of global solutions of some nonlinear space-nonlocal evolution equations on the Heisenberg group, {\it Electron. J. Differ. Equ.}, 2015(227), 1–10 (2015).
			
			\bibitem{AZ22}		J.-Ph. Anker and H.-W. Zhang, Wave equation on general noncompact symmetric spaces, (to appear in) {\em Amer. J. Math}, 2022. arXiv:2010.08467
			
			
			\bibitem{EE1} D. Applebaum, L\'evy processes from probability to finance quantum groups, \emph{Notices Amer. Math. Soc.} 51, 1336-1347 (2004).
			
			\bibitem{EE3}  G. Autuori and P. Pucci, Elliptic problems involving the fractional Laplacian in $\mathbb{R}^n $, \emph{J.
				Differential Equations} 255(8),  2340-2362 (2013).
			
			\bibitem{BKM22} A. K. Bhardwaj, V. Kumar, and S. S, Mondal, Estimates for the nonlinear viscoelastic damped wave equation on compact Lie groups, arXiv.2207.06645  (2022). 
			
			\bibitem{EE4}  G. M. Bisci and V. D.  Rădulescu, Applications of local linking to nonlocal Neumann problems, 
			\emph{Commun. Contemp. Math.} 17(1),    1450001 (2015).
			
			\bibitem{EE5}  G. M. Bisci and V. D.  Rădulescu, Ground state solutions of scalar field fractional Schr\"odinger equations, \emph{Calc. Var. Partial Differential Equations} 54,  2985-3008 (2015).
			
			\bibitem{EE6}  G. M. Bisci, V. D.  Rădulescu,  and R. Servadei, {\it Variational methods for nonlocal fractional	problems, Encyclopedia of Mathematics and its Applications}, 162, Cambridge University Press, Cambridge (2016).
			
			\bibitem{BGT04} N. Burq, P. Gerard, and N. Tzvetkov,  Strichartz inequalities and the nonlinear Schrödinger equation on
			compact manifolds, {\it Am. J. Math.} 126(3), 569–605 (2004)
			
			\bibitem{EE9} L. Caffarelli and L. Silvestre, An extension problem related to the fractional Laplacian, \emph{Comm. Partial Differential Equations} 32, 1245-1260 (2007).
			
			\bibitem{EE10}  M. Caponi and P. Pucci, Existence theorems for entire solutions of stationary Kirchhoff fractional $p$-Laplacian equations, \emph{Ann. Mat. Pura Appl.} 195, 2099-2129 (2016).
			
			
			
			
			
			\bibitem{shyamm}	 A. Dasgupta, V. Kumar,  and S. S.    Mondal, Nonlinear fractional wave equation on compact Lie groups, 	arXiv:2207.04422 (2022).
			
			
			
			
			
			\bibitem{f3} J. L. A. Dubbeldam, A. Milchev, V. G. Rostiashvili, and T. A. Vilgis, Polymer translocation through a nanopore: a showcase of anomalous diffusion, \emph{Phys. Rev. E} 76, 010801 (2007).
			
			
			
			
			\bibitem{Rei}  M. R. 	Ebert and M.  Reissig,  Methods for Partial Differential Equations, Birkh\"auser, Basel (2018).
			
			\bibitem{GLS97} V, Georgiev, H. Lindblad, and C. D. Sogge,  Weighted Strichartz estimates and global existence for semi-
			linear wave equations, {\it Am. J. Math.} 119(6), 1291–1319 (1997).
			
			\bibitem{garetto} 	C.	 Garetto and M.  Ruzhansky, Wave equation for sums of squares on compact Lie groups, \emph{J. Differential Equations} 258(12),  4324-4347 (2015).		
			
			\bibitem{f4} R. Herrmann, {\it  Fractional Calculus: An Introduction for Physicists}, World Scientific,  Singapore (2011).
			
			\bibitem{f10} R. Hilfer, {\it Applications of Fractional Calculus in Physics}, World Scientific, Singapore (2000). 
			
			
			\bibitem{IKeta and Tanizawa}  R. Ikehata and K. Tanizawa, Global existence of solutions for semilinear damped wave equations in $\mathbb{R}^n$ with noncompactly supported initial data, \emph{Nonlinear Anal.} 61(7), 1189-1208  (2005).
			
			\bibitem{KAP95} L. Kapitanski, Minimal compact global attractor for a damped semilinear wave equation, {\it  Commun.
				Partial. Differ. Equ.} 20(7/8), 1303–1323 (1995).
			
			\bibitem{EE27} N. Laskin, Fractional Schr\"oodinger equation, \emph{Phys. Rev. E}, 66,  056108 (2002).
			
			\bibitem{f11} Yu. Luchko and A. Punzi, Modeling anomalous heat transport in geothermal reservoirs via fractional diffusion equations,
			\emph{GEM Int. J. Geomath.} 1,  257-276 (2011).
			
			
			
			\bibitem{Matsumura}  A. Matsumura, On the asymptotic behavior of solutions of semi-linear wave equations, \emph{Publ. Res. Inst. Math. Sci.} 12(1), 169-189 (1976/77).
			
			
			
			\bibitem{24} A. I. Nachman, The wave equation on the Heisenberg group, \emph{Comm. Partial Differential Equations} 7(6),   675-714 (1982).
			
			\bibitem{EE24} E. Di Nezza, G. Palatucci, and E. Valdinoci, Hitchhiker's Guide to the fractional Sobolev spaces, \emph{Bull. Sci. Math.} 136, 521-573  (2012).					
			
			
			
			\bibitem{NN3} K. B. Oldham and J. Spanier, {\it The fractional calculus: Theory and applications of differentiation and integration to arbitrary order}, Academic Press, London (1974).
			
			\bibitem{27}   A. Palmieri,  On the blow-up of solutions to semilinear damped wave equations with power nonlinearity in compact Lie groups, \emph{J. Differential Equations } 281, 85-104 (2021). 
			
			\bibitem{31}      A.  Palmieri, Semilinear wave equation on compact Lie groups, \emph{J. Pseudo-Differ. Oper. Appl.} 12, 43 (2021).
			
			\bibitem{28}  A.  Palmieri,  A global existence result for a semilinear wave equation with lower order terms on compact Lie groups,  \emph{J. Fourier Anal. Appl. }28,  Article number: 21 (2022).
			
			
			
			
			\bibitem{NN4} I. Podlubny, {\it Fractional differential equations}, Academic press, New York   (1999).
			
			\bibitem{30} M. Ruzhansky and N. Tokmagambetov, Nonlinear damped wave equations for the sub-Laplacian on the Heisenberg group and for Rockland operators on graded Lie groups, \emph{J. Differential Equations} 265(10), 5212-5236 (2018).
			
			\bibitem{gra1}   M. Ruzhansky and C. Taranto, Time-dependent wave equations on graded groups, \emph{Acta Appl. Math.} 171, Article number: 21 (2021).
			
			\bibitem{RuzT} M. Ruzhansky and V. Turunen, \emph{ Pseudo-differential Operators and Symmetries: Background Analysis and Advanced Topics}, Birkha\" user-Verlag, Basel (2010).
			
			
			
			
			
			\bibitem{RT13} M. Ruzhansky and V. Turunen, Global quantization of pseudo-differential operators on compact Lie groups, $SU(2),$ $3$-Sphere, and homogeneous spaces, \emph{Int. Math. Res. Not. IMRN} (11), 2439-2496 (2013).
			
			
			
			
			
			\bibitem{Gall}    M. Ruzhansky and  N. Yessirkegenov, Hardy, Hardy-Sobolev, Hardy-Littlewood-Sobolev and Caffarelli-Kohn-Nirenberg inequalities on general Lie groups, arXiv:1810.08845 (2019).
			
			
			
			
			
			
			\bibitem{gra2}   M. Ruzhansky and N. Yessirkegenov, Very weak solutions to hypoelliptic wave equations, \emph{J. Differential Equations}  268(5), 2063-2088 (2020).
			
			
			
			\bibitem{Neww1} S. G. Samko, A. A. Kilbas, and  O. I. Marichev,   {\it  Fractional Integrals and Derivatives of the fractional-order and their some applications}, Minsk: Nauka i Tekhnika   (1987). 
			
			\bibitem{EE31} R. Servadei and E. Valdinoci, Mountain pass solutions for non-local elliptic operators, \emph{J. 	Math. Anal. Appl.} 389, 887-898 (2012).
			
			
			
			\bibitem{SSW19} Y. Sire, C. D. Sogge, and C. Wang, The Strauss conjecture on negatively
			curved backgrounds, {\em Discrete Contin. Dyn. Syst.}, 39:7081–7099, 2019.
			
			
			
			
			
			\bibitem{gra3}   C. Taranto, Wave equations on graded groups and hypoelliptic Gevrey spaces, Imperial College London Ph.D. thesis, 2018,	arXiv:1804.03544 (2018). 
			
			
			\bibitem{Todorova}  G. Todorova and  B. Yordanov, Critical exponent for a nonlinear wave equation with damping, \emph{J. Differ. Equ.} 174(2), 464-489 (2001).
			
			\bibitem{NN1} P. J.  Torvik and R. L. Bagley, On the appearance of the fractional derivative in the behavior of real materials, \emph{J. Appl. Mech.} 51(2), 294-298 (1984).
			
			
			
			
			
			
			
			
			
			
			
			
			
			
			
			
			
			
			
			
			
			
			
			
			
			
			
			
			
			\bibitem{Zhang} Q. S. Zhang, A blow-up result for a nonlinear wave equation with damping: the critical case, \emph{C. R. Acad. Sci. Paris, Ser. I} 333(2), 109-114  (2001).
			
			\bibitem{Zhang20}
			H.-W. Zhang, Wave and Klein-Gordon equations on certain locally symmetric spaces, \emph{J. Geom. Anal.}, 30(4):4386--4406, 2020.
			\bibitem{Zhang21} 
			H.-W. Zhang, Wave equation on certain noncompact symmetric spaces, {\em Pure Appl.
				Anal.} 3:363-386, 2021.
		\end{thebibliography}
	\end{document}